\documentclass[14pt]{amsart}
\usepackage{cases}
\usepackage{amsmath}
\usepackage{amsfonts}
\usepackage{bm}
\usepackage{amsfonts,amsmath,amssymb,amscd,bbm,amsthm,mathrsfs,dsfont}
\usepackage{mathrsfs}
\usepackage{pb-diagram}
\usepackage{color}
\usepackage{amssymb}
\usepackage{xypic}
\usepackage[all]{xy}
\usepackage{mathrsfs}
\usepackage{amsthm}
\usepackage[all]{xy}
\usepackage{epsfig}
\usepackage{graphics}
\usepackage{array}
\usepackage{graphicx} 
\usepackage{epstopdf}
\usepackage{float}
\usepackage{tikz}
\usepackage{tikz-cd}
\usetikzlibrary{arrows}
\usetikzlibrary{graphs}
\usepackage{tikz}
\usepackage{hyperref}

\newcommand{\ZK}{(}
\newcommand{\YK}{)}
\newcommand{\DP}{,}
\newcommand{\PL}{+}

\newtheorem{Theorem}{Theorem}[section]
\newtheorem{Lemma}[Theorem]{Lemma}
\newtheorem{Definition}[Theorem]{Definition}
\newtheorem{Corollary}[Theorem]{Corollary}
\newtheorem{Proposition}[Theorem]{Proposition}
\newtheorem{Example}[Theorem]{Example}
\newtheorem{Remark}[Theorem]{Remark}

\newtheorem{Problem}[Theorem]{Problem}

\title {Periodicities in cluster algebras and  cluster automorphism groups}

\author{Fang Li $\;\;\;\;\;\;$ Siyang Liu $\;\;\;\;\;\;$}
\address{Fang Li
\newline Department of Mathematics, Zhejiang University (Yuquan Campus), Hangzhou, Zhejiang
310027,  P.R.China}
\email{fangli@zju.edu.cn}

\address{Siyang Liu
\newline Department
of Mathematics, Zhejiang University (Yuquan Campus), Hangzhou, Zhejiang
310027, P.R.China}
\email{siyangliu@zju.edu.cn}

\date{version of \today}

\newcommand{\lra}{\longrightarrow}

\newcommand{\ra}{\rightarrow}
\newcommand{\sdp}{\times\kern-.2em\vrule height1.1ex depth-.05ex}
\newcommand{\epi}{\lra \kern-.8em\ra}

\makeatletter
\@addtoreset{equation}{section}
\makeatother

 \normalbaselines
\begin{document}
\renewcommand{\thefootnote}{\alph{footnote}}
\setcounter{footnote}{-1} 
\renewcommand{\thefootnote}{\alph{footnote}}
\setcounter{footnote}{-1} 

\renewcommand{\thefootnote}{\alph{footnote}}
\maketitle
\bigskip
\begin{abstract}
In this article, we study the relations between groups related to cluster automorphism groups
 which are defined by Assem, Schiffler and Shamchenko in \cite{ASS}. We establish  the relationship among
 (strict) direct cluster automorphism groups and those groups consisting of periodicities of respectively labeled seeds and exchange matrices in the language of short exact sequences. As an application, we characterize automorphism-finite
 cluster algebras in the cases with bipartite seeds or finite mutation type. Finally, we study the relation between
 the groups $\mathrm{Aut}\mathcal{A}$ and $\mathrm{Aut}_{M_n}S$ and give the negative
 answer via counter-examples to King and Pressland's a problem in \cite{KP}.
\end{abstract}

\section{Introduction}
Cluster algebras were invented by Fomin and Zelevinsky in a series of papers  \cite{FZ1,FZ2,FZ4}. They are defined as
commutative $\mathbb{Z}$-algebras generated by cluster variables. Many relations between cluster
algebras and other branches of mathematics have been discovered, such as periodicities of $T$-systems
and $Y$-systems, representations of quivers, combinatorics.  There are two important kinds of cluster algebras,
which are of finite type and finite mutation type. Cluster algebras with finite cluster variables
are said to be of finite type. This kind of cluster algebras has been classified via Dynkin graphs by
Fomin and Zelevinsky in \cite{FZ3}. Cluster algebras with finite exchange matrices are said to be of finite
mutation type, which are classified by Felikson,  Shapiro, and Tumarkin in \cite{FST1,FST2}.

Cluster automorphisms  are $\mathbb{Z}$-automorphisms of cluster algebras with trivial coefficients and with
skew-symmetric exchange matrices,  which commute with mutations. They were firstly defined by  Assem,
Schiffler and Shamchenko in \cite{ASS} and studied by a lot of authors. Chang and Zhu  studied the cluster
automorphism group of a skew-symmetric cluster algebra with geometric coefficients in \cite{CZ0}.
They extended cluster automorphisms to cluster algebras with skew-symmetrizable exchange matrices of
finite type in \cite{CZ2}, and  showed cluster automorphism groups of cluster algebras of finite type have close
relations with the so-called $\tau$-transformation group.
They also studied the relation between cluster automorphism groups of skew-symmetrizable cluster algebras of finite type
or skew-symmetric algebras of finite mutation type and the corresponding automorphism groups of exchange graphs in \cite{CZ1}.
Lawson extended some results in \cite{CZ1} to cluster algebras with skew-symmetrizable exchange matrices of finite mutation type by introducing a
marking on the exchange graph in \cite{Law}. The further related study on this topic was given for sign-skew-symmetric
cluster algebras in \cite{HLY}.

Periodicities in cluster algebras were firstly introduced and studied by Fomin and Zelevinsky in \cite{FZ3}. They proved
periodicity conjecture of Zamolodchikov on $Y$-systems from indecomposable Cartan matrices of finite type.
Nakanishi studied periodicities in cluster algebras in general cases rather than finite type in \cite{Nak}.
He also proved Restriction and Extension Theorem for cluster algebras with coefficients
from a subtraction-free semifield. In this paper, we define two groups consisting of mutations associated to
periods of exchange matrices and labeled seeds, respectively. We show that quotients of the two groups have
 close relations with strict direct cluster automorphism groups.
We formulate relationships between strict direct cluster automorphism groups,
direct cluster automorphism groups and permutation-periodic groups. These relations are given by
the following two short exact sequences of groups.
\begin{Theorem}[Theorem \ref{exa1} and \ref{thm2}]
Let $(\mathbf{x},B)$ be a labeled seed of a cluster algebra $\mathcal{A}$, then there is an exact sequence
\[ 1\rightarrow H(\mathbf{x})\rightarrow G(B)\xrightarrow{\varphi} \mathrm{SAut}^+(\mathcal{A})\rightarrow 1.  \]
 \[ 1\rightarrow \mathrm{SAut}^+(\mathcal{A}) \rightarrow \mathrm{Aut}^+(\mathcal{A}) \xrightarrow{\phi} L_n/P_n \rightarrow 1.\]
\end{Theorem}

In \cite{ASS}, Assem, Schiffler
and Shamchenko  defined automorphism-finite cluster algebras and
proved that for any skew-symmetric cluster algebra which is acyclic or from a surface, it is
automorphism-finite if and only if
it is of finite type, i.e., its exchange matrix is mutation equivalent to an skew-symmetric matrix of Dynkin type.
As an application, we prove that a skew-symmetrizable cluster algebra with bipartite seeds or of finite mutation type
is automorphism-finite if and only if it is of finite type by considering strict direct cluster automorphisms,
which are $\mathbb{Z}$-automorphisms of cluster algebras commuting with mutation of labeled seeds.

On the other hand, in \cite{KP}, King and Pressland introduced the labeled mutation class $S$ of  a
skew-symmetric labeled seed and the group $\mathrm{Aut}_{M_n}S$ consisting of bijections from
$S$ to $S$ which commute with the action of the mutation group $M_n$. They proved  that the cluster
automorphism group is embedded into the group $\mathrm{Aut}_{M_n}S$ and that for finite mutation type case,
this embedding is in fact an isomorphism, see Theorem \ref{KP1} or [\cite{KP}, Corollary 6.3,6.4].
However for skew-symmetric  cluster algebras of infinite mutation type, they are not sure if the conclusion still holds, which is
proposed as a problem, see Problem \ref{prob}. In Section \ref{final}, we firstly give some
sufficient conditions for the embedding to be an isomorphism. Finally
we answer this problem by giving a class of cluster algebras
of infinite mutation type which satisfies the condition in Problem \ref{prob}.

The organization of the paper is as follows. In Section \ref{preli}, we introduce the basic
notions for cluster algebras and cluster automorphisms. In Section \ref{3}, we recall basic concepts
of periodicities of labeled seeds and exchange matrices following from \cite{Nak} and define mutation-periodic
groups and permutation-periodic groups for exchange matrices and labeled seeds, and their relations
 with (strict) direct cluster automorphism groups are given.  In Section \ref{4}, we show that
for skew-symmetrizable cluster algebras with bipartite seeds or of finite mutation type,
the cluster automorphism group is finite if and only if the cluster algebra is of finite type.
Finally we give some sufficient conditions for cluster algebras to satisfy $\mathrm{Aut}\mathcal{A}=\mathrm{Aut}_{M_n}S$
in Section \ref{final}, and then answer Problem \ref{prob}.

In this paper, we assume that cluster algebras are always skew-symmetrizable. In particular,
 in Section \ref{final}, we only consider those in skew-symmetric case.

\section{Preliminaries} \label{preli}

\subsection{Cluster algebras}
In this subsection, we recall basic concepts and important properties of cluster algebras. In this paper, we focus
on cluster algebras without coefficients  (that is, with trivial coefficients).
For a positive integer $n$, we will always denote by $[1, n]$ the set $\{1,2,\dots, n\}$.

Take the ambient field $\mathcal{F}$ isomorphic to the field of rational functions
in $n$ independent variables with coefficients in $\mathbb{Q}$.
A {\bf labeled seed} is a pair $(\mathbf{x}, B)$ in which $\mathbf{x}$ is an $n$-tuple of free
generators of $\mathcal{F}$, and $B$ is an $n\times n$ skew-symmetrizable integer matrix.
Recall that $B$ is said to be skew-symmetrizable if there exists an positive definite diagonal integer
matrix $D$ such that $DB$ is skew-symmetric. For $k\in [1,n]$, define another pair
$(\mathbf{x'},B') = \mu_k(\mathbf{x},B)$ which is called the mutation of $(\mathbf{x},B)$
 at $k$ and obtained by the following rules:
\begin{enumerate}
\item[(1)] $\mathbf{x'} = (x'_1, \dots, x'_n)$ is given by
\[x'_k = \frac{\prod x_i^{[b_{ik}]_+} + \prod x_i^{[-b_{ik}]_+}}{x_k}\] and $x'_i = x_i$ for $i \neq k$;
\item[(2)] $B'=\mu_k(B)=(b'_{ij})_{n\times n}$ is given by \[b'_{ij}=\begin{cases} -b_{ij},&\text{if $i=k$ or $j=k$;}\\b_{ij} + \mathrm{sgn}(b_{ik})[b_{ik}b_{kj}]_+,&\text{otherwise,}\end{cases}\]
\end{enumerate}
where $[x]_+=\mathrm{max}\{x,0\}$. Note that $(\mathbf{x'},B')$ is also a labeled
seed and $\mu_k$ is an involution. In a labeled seed $(\mathbf{x},B)$, where $\mathbf{x}=(x_1,x_2,\dots,x_n)$ $\mathbf{x}$ and $B=(b_{ij})_{n\times n}$,
is called a labeled cluster, elements in $\mathbf{x}$ are called {\bf cluster variables}, and $B$ is
called an {\bf exchange matrix}. The unlabeled seeds are obtained by identifying labeled seeds that differ from each other by simultaneous permutations
of the components in $\mathbf{x}$, and of the rows and columns of $B$. We will refer to unlabeled seeds simply as seeds, when there is no risk of confusion.
 Throughout this paper, without loss of generality, we always assume that $B$ is indecomposable as a matrix, i.e.,
for any $i,j$, there exist $i=i_0, i_1, \dots, i_k=j$ such that $b_{i_0i_1}b_{i_1i_2}\dots b_{i_{k-1}i_k}\neq 0$.
\begin{Definition}[\cite{FZ1,FZ4}]

\begin{enumerate}
\item[(1)] Two labeled seeds  $(\mathbf{x},B)$ and $(\mathbf{x}',B')$ are said to be {\bf mutation equivalent} if there is a finite mutation sequence exchanging $(\mathbf{x},B)$ to $(\mathbf{x}',B')$;

 \item[(2)] Two labeled seeds  $(\mathbf{x},B)$ and $(\mathbf{x}',B')$ are said to be {\bf equivalent} (or say, they define the {\bf same unlabeled seed})
if $(\mathbf{x}',B')$ can be  obtained from $(\mathbf{x},B)$ by simultaneously re-labeling $n$-tuple $\mathbf{x}$ and the corresponding
re-labeling of the rows and columns of $B$.
\end{enumerate}
\end{Definition}

Note that two labeled seeds may not mutation equivalent as labeled seeds even though they are the same unlabeled seeds.

\begin{Definition}[\cite{FZ4}]
Let $\mathrm{T}_n$ be an $n$-regular tree and valencies emitting from each vertex are labelled by $1,2,\dots,n$.
A cluster pattern is an $n$-regular tree $\mathrm{T}_n$ such that for each vertex $t\in \mathrm{T}_n$, there is a
labeled seed $\Sigma_t=(\mathbf{x}_t,B_t)$ and for each edge labelled by $k$, two labeled seeds in the endpoints
are obtained from each other by seed mutation at $k$.  And $\Sigma_t=(\mathbf{x}_t,B_t)$ are written as follows:
\[\mathbf{x}_t=(x_{1,t},x_{2,t},\dots,x_{n,t}),\,\,\,\,\,\,B_t=(b_{ij}^t).\]
\end{Definition}

Note that a cluster pattern is uniquely determined by one labeled seed, thus for a labeled seed $(\mathbf{x},B)$,
we may associate with a cluster pattern $T_n(\mathbf{x},B)$. The cluster algebra $\mathcal{A}=\mathcal{A}(\mathbf{x}_{t_0},B_{t_0})$
associated to the initial seed $(\mathbf{x}_{t_0},B_{t_0})$ is a $\mathbb{Z}$-subalgebra of $\mathcal{F}$
generated by cluster variables appeared in $T_n(\mathbf{x}_{t_0},B_{t_0})$. One of the most important properties
in cluster algebras is the Laurent phenomenon, which says any cluster variable can be expressed as a Laurent
polynomial in terms of cluster variables in the initial labeled seed with coefficients in $\mathbb{Z}$.
These Laurent polynomials are conjectured to have positive coefficients, which had been proved for skew-symmetric
cluster algebras in \cite{LS} and for the skew-symmetrizable cluster algebras  in \cite{GHKK}.

A cluster algebra arising from a labeled seed with skew-symmetric (skew-symmetrizable, resp.)
exchange matrix is also called skew-symmetric (skew-symmetrizable, resp.). There is a bijection
between skew-symmetrizable matrices and valued quivers. Indeed for a skew-symmetrizable matrix
$B=(b_{ij})_{n\times n}$, we define a valued quiver $(Q,v)$ as follows. The vertices set $Q_1$ is
given by $\{1,2,\dots,n\}$, and there is an arrow $\alpha$ from $i$ to $j$ whenever $b_{ij}>0$
and the value $v(\alpha)$ is defined to be $(|b_{ij}|,|b_{ji}|)$. There is also a bijection between
skew-symmetric matrices and cluster quivers, which are finite quivers without loops nor directed $2$-cycles.
For a skew-symmetric matrix $B=(b_{ij})_{n\times n}$, the quiver $Q$ of $B$ is defined as follows.
The vertices of $Q$ is $\{1,2,\dots\}$ and there are $b_{ij}$ arrows from $i$ to $j$ if $b_{ij}>0$
and there are no arrows from $i$ to $j$ if $b_{ij}\leqslant 0$. In this paper, quivers are assumed to be
cluster quivers. Mutation of quivers are defined as follows.
\begin{Definition}[\cite{FZ1}]
Let $Q$ be a quiver and $k\in Q_0$ be a fixed vertex. The mutation $\mu_k(Q)$ of $Q$ at $k$ is obtained by the following steps:
\begin{enumerate}
\item[(1)] For every 2-path $i\rightarrow k\rightarrow j$, add a new arrow $i\rightarrow j$;
\item[(2)] Reverse all arrows incident with $k$;
\item[(3)] Delete a maximal collection of 2-cycles from those created in $(1)$.
\end{enumerate}
\end{Definition}
Note that mutation of quivers is compatible with mutation of skew-symmetric matrices, i.e., let $Q$ and
$Q'$ be the corresponding quivers of $B$ and $\mu_k(B)$, then $Q'=\mu_k(Q)$ and vice versa. The so-called
weighted quiver is obtained from a quiver by replacing its multiple arrows by a single arrow and assign
it the value given by the multiplicity of the arrows. For simplicity, we sometimes use weighted quivers to replace
quivers. If the weight of an arrow is $1$, we usually omit it.
\begin{Example}
The quiver and weighted quiver of the skew-symmetric matrix
\[B=\begin{bmatrix} 0&2&1\\-2&0&1\\-1&-1&0 \end{bmatrix}. \]
 are given as follows respectively:
\[\mathord{\begin{tikzpicture}[scale=1.3,baseline=0]
\node at (0,0.3) (12) {$2$};
\node at (-1,-0.3) (11) {$1$};
\node at (1,-0.3) (13) {$3$,};
\node at (3,0.3) (22) {$2$};
\node at (2,-0.3) (21) {$1$};
\node at (4,-0.3) (23) {$3$.};
\path[-angle 90]
	(11) edge [bend left]  (12)
    (11) edge   (12)
    (12) edge   (13)
    (11) edge   (13)
	(21) edge node [left] {$2$} (22)
	(22) edge  (23)
	(21) edge  (23);
\end{tikzpicture}}\]

\end{Example}

\begin{Proposition}[\cite{GSV}]\label{clu}
Every seed is uniquely determined by its cluster, i.e., for two mutation equivalent seeds $(\mathbf{x},B)$ and $(\mathbf{x}',B')$,
if $x_i'=x_{\sigma(i)}$ for some $\sigma\in S_n$ and any $i\in [1,n]$, then $b'_{ij}=b_{\sigma(i)\sigma(j)}$ for any $i, j\in [1,n]$.
\end{Proposition}
This proposition shows that in a (labeled) seed $(\mathbf{x}',B')$, the exchange matrix $B'$ is uniquely determined by the (labeled)
cluster $\mathbf{x}'$. We may use $B(\mathbf{x}')$ to denote the corresponding exchange matrix of $\mathbf{x}'$, and we sometimes use the
cluster $\mathbf{x}$ to denote the labeled seed $(\mathbf{x},B)$.

\subsection{Finite type and finite mutation type}
Let $(\mathbf{x},B)$ be a labeled seed, the cluster algebra $\mathcal{A}=\mathcal{A}(\mathbf{x},B)$ is said to be of finite type
if the set consisting of cluster variables appeared in $T_n(\mathbf{x},B)$ is a finite set, which is equivalent to say there are
finite clusters mutation equivalent to $(\mathbf{x},B)$.
Cluster algebras of finite type have been classified by Fomin and Zelevinsky in \cite{CZ2}. They showed
this classification is identical to the Cartan-Killing classification of semisimple Lie algebras and finite root systems.
Indeed, it is showed that the cluster algebra $\mathcal{A}=\mathcal{A}(\mathbf{x},B)$ is of finite type if and
only if $\Gamma(B)$ is equivalent to an oriented Dynkin graph in \cite{CZ2}, where the diagram $\Gamma(B)$ is a weighted quiver
associated to $B$ whose vertex set is $[1,n]$ and there is an arrow $i\rightarrow j$ with weight $|b_{ij}b_{ji}|$ if and only if $b_{ij}>0$ .
If $\mathcal{A}=\mathcal{A}(\mathbf{x},B)$ is of finite type, we also call $(\mathbf{x},B)$ and $B$ are of finite type.
In particular, if $B$ is skew-symmetric, then it is of finite type if and only if its quiver is mutation equivalent
to an orientation of a Dynkin graph.

The cluster algebra $\mathcal{A}=\mathcal{A}(\mathbf{x},B)$ is said to be of finite mutation type
if the set consisting of exchange matrices appeared in $T_n(\mathbf{x},B)$ is a finite set. Cluster algebras of
finite mutation type have been classified by Felikson,  Shapiro, and Tumarkin  in \cite{FST1,FST2}.
The cluster algebra $\mathcal{A}(\mathbf{x},B)$ is of finite mutation type if and only if for any $i,j\in [1,n]$, $|b_{ij}'b_{ji}'|\leqslant 4$ holds
for any $B'$ mutation equivalent to $B$. For skew-symmetric
cases, the cluster algebra $\mathcal{A}=\mathcal{A}(\mathbf{x},B)$ is of finite mutation type if and only if $B$ is arising
from a triangulation of a surface, or a generalized Kronecker quiver, or other $11$ exceptional quivers listed in  \cite{FST1}.
For skew-symmetrizable case, Felikson,  Shapiro, and Tumarkin  classified cluster algebras of finite mutation
type via unfolding in \cite{FST2}.

The cluster algebra $\mathcal{A}=\mathcal{A}(\mathbf{x},B)$ is said to be acyclic
if there exists an exchange matrix in $T_n(\mathbf{x},B)$ whose (valued) quiver is acyclic.
In this case, the matrix $B$ is said to be mutation-acyclic. However if for any exchange matrix
in $T_n(\mathbf{x},B)$, its (valued) quiver is not acyclic, the matrix $B$ is called mutation-cyclic.

It is obvious that if $B$ is of finite type, then it is mutation-acyclic and it is also of finite
mutation type. The converse is not true. Actually an acyclic quiver is of finite mutation type if
and only if it is an orientation of a Dynkin graph or an extended Dynkin graph. Notice also that not all quivers of finite
mutation type are mutation-acyclic.
\subsection{Cluster automorphisms}
In this subsection, basic concepts of cluster automorphisms are recalled. We also introduce strict direct cluster
automorphisms for cluster algebras without coefficients. We consider relations between periods of labeled seeds and
exchange matrices and strict direct cluster automorphisms. As an application, we consider cluster automorphism
finite cluster algebras.

Recall that for any (labeled) seed, the exchange matrix
is uniquely determined by the cluster (see \ref{clu}). It is suitable to represent a seed by its cluster in the following.
\begin{Definition}[\cite{ASS}]\label{cluauto} Let $\mathcal{A}=\mathcal{A}(\mathbf{x},B)$ be a cluster algebra, and $f:\mathcal{A}\rightarrow \mathcal{A}$
be an automorphism of $\mathbb{Z}$-algebras. If there is a seed $(\mathbf{x'},B')$ of $\mathcal{A}$ such that

(1) $f(\mathbf{x'})$ is a cluster;

(2) $f$ is compatible with mutations, i.e., for every $x\in \mathbf{x}$, we have \[f(\mu_{x,\mathbf{x}'}(x))=\mu_{f(x),f(\mathbf{x}')}(f(x)),\]
then $f$ is called a cluster automorphism of $\mathcal{A}$.
\end{Definition}

Note that in Definition \ref{cluauto}, the seed is an unlabeled seed mutation equivalent to the initial seed
$(\mathbf{x},B)$. There are some equivalent conditions for an automorphism of a cluster algebra to be a cluster
automorphism.

\begin{Proposition}[\cite{ASS}] Let $f$ be a $\mathbb{Z}$-algebra automorphism of $\mathcal{A}$. Then the following conditions
are equivalent:
\begin{enumerate}
\item[(i)] $f$ is a cluster automorphism of $\mathcal{A}$;
\item[(ii)] $f$ satisfies (1)(2) in Definition \ref{cluauto} for every seed;
\item[(iii)] $f$ maps each cluster to a cluster;
\item[(iv)] there exists a seed $(\mathbf{x'},B')$ such that $f(\mathbf{x}')$ is a cluster, and $B(f(\mathbf{x}'))=B'$ or $-B'$.
\end{enumerate}
\end{Proposition}
\begin{Remark}\label{rm2}
For any direct cluster automorphism $f:\mathcal{A}\rightarrow \mathcal{A}$, there exists a mutation
sequence $\mu_{\mathbf{i}}$ and a permutation $\sigma\in S_n$ such that $f(\mathbf{x})=\sigma(\mu_{\mathbf{i}}(\mathbf{x}))$
and $B(\mu_{\mathbf{i}}(\mathbf{x}))=B^{\sigma^{-1}}$ by definitions.
\end{Remark}
\begin{Corollary}[\cite{ASS}]
Let $f:\mathcal{A}\rightarrow \mathcal{A}$ be a cluster automorphism of $\mathcal{A}$. Fix a seed $(\mathbf{x'},B')$
satisfying $B(f(\mathbf{x}'))=B'$ or $-B'$.
Then
\begin{enumerate}
\item[(i)] if $B(f(\mathbf{x}'))=B'$, then for any seed $(\mathbf{x''},B'')$ of $\mathcal{A}$, we have $B(f(\mathbf{x}''))=B''$,
\item[(ii)] if $B(f(\mathbf{x}'))=-B'$, then for any seed $(\mathbf{x''},B'')$ of $\mathcal{A}$, we have $B(f(\mathbf{x}''))=-B''$.
\end{enumerate}
\end{Corollary}

Following from these results from \cite{ASS}, cluster automorphisms are classified into two kinds:
one $f$ is called {\bf direct cluster automorphism} if it satisfies $B(f(\mathbf{x'}))=B'$ for any $(\mathbf{x'},B')$,
the other one $f$ is called {\bf inverse cluster automorphism} if it satisfies $B(f(\mathbf{x}))=-B$ for any $(\mathbf{x'},B')$.
Let $\mathrm{Aut}(\mathcal{A})$, $\mathrm{Aut}^+(\mathcal{A})$, and
$\mathrm{Aut}^-(\mathcal{A})$ denote the sets of all cluster automorphisms, direct cluster automorphisms and inverse cluster
automorphisms, respectively. Obviously, they form groups under compositions respectively. Moreover, these groups have following properties.

\begin{Proposition}[\cite{ASS}]\label{cor1} The direct cluster automorphism group $\mathrm{Aut}^+(\mathcal{A})$ is a normal subgroup of $\mathrm{Aut}(\mathcal{A})$
of index at most two. Therefore
the direct cluster automorphism group $\mathrm{Aut}^+(\mathcal{A})$ is a finite group if and only if $\mathrm{Aut}(\mathcal{A})$
is a finite group.
\end{Proposition}

Cluster automorphisms are defined above via unlabeled seeds. However labeled seeds are more relevant to periods in
cluster algebras, this motivates us to define strict cluster automorphisms by labeled seeds as special cluster
automorphisms in this section.
\begin{Definition}\label{strcluauto} Let $\mathcal{A}=\mathcal{A}(\mathbf{x},B)$ be a cluster algebra, and $f:\mathcal{A}\rightarrow \mathcal{A}$
be an automorphism of $\mathbb{Z}$-algebras. If there is a labeled seed $(\mathbf{x'},B')$ of $\mathcal{A}$ such that
\begin{enumerate}
\item[(1)]  $f(\mathbf{x'})$ is a labeled cluster, i.e., there exists an mutation sequence $\mu_{i_k}\dots\mu_{i_2}\mu_{i_1}$ such that
$(f(\mathbf{x'}), B(f(\mathbf{x'})))=\mu_{i_k}\dots\mu_{i_2}\mu_{i_1}(\mathbf{x'},B')$ as labeled seeds,

\item[(2)] $f$ is compatible with mutations,
\end{enumerate}
then f is called a \textbf{strict cluster automorphism} of $\mathcal{A}$.
\end{Definition}
Similarly, if $f$ is strict cluster automorphism such that $B(f(\mathbf{x}))=B$, then $f$ is called a {\bf strict
direct cluster automorphism} of $\mathcal{A}$. Let $\mathrm{SAut}^+(\mathcal{A})$ be the set of all strict direct cluster
automorphisms of $\mathcal{A}$. It is obvious that $\mathrm{SAut}^+(\mathcal{A})$ is a subgroup of $\mathrm{Aut}^+(\mathcal{A})$.

\section{Mutation-periodic groups and permutation-periodic groups}\label{3}
\subsection{Periodicities in cluster algebras}We recall some basic definitions and properties on periodicities in cluster algebras in this subsection.
Let $I$ be a subset of $[1,n]$. An ordered sequence $\mathbf{i}=(i_1,i_2,\dots,i_s)$ is called an $I$-sequence if
$i_p\in I$ for any $p\in [1,s]$. Moreover, $\mathbf{i}$ is called an essential $I$-sequence if it is
an $I$-sequence and satisfies $i_p\neq i_{p+1}$ for any $p\in [1,n-1]$.
For any $I$-sequence $\mathbf{i}=(i_1,i_2,\dots,i_s)$, we denote by $\mathbf{i}^{-1}$ the
$I$-sequence $(i_s,i_{s-1},\dots,i_1)$, and we define $\mu_{\mathbf{i}}$ to be the composition
$\mu_{i_s}\dots \mu_{i_2}\mu_{i_1}$ of mutations. Note that $\mu_{\mathbf{i}}\mu_{\mathbf{i}^{-1}}=\mathrm{id}$
since every $\mu_k$ is an involution.

For a labeled seed  $(\mathbf{x},B)$  of rank $n$  and a permutation $\sigma \in S_n$, we
define the action  of $\sigma$ satisfying that
$\sigma(\mathbf{x},B)$(or written as $(\mathbf{x},B)^{\sigma} :=(\mathbf{x}^{\sigma},B^{\sigma})$,
where $\mathbf{x}^{\sigma}=(x_{\sigma(i)})_{i\in I},B^{\sigma}=(b_{\sigma(i)\sigma(j)})_{i,j\in I}.$

\begin{Definition}[\cite{Nak}]  Let $\mathcal{A}=\mathcal{A}(\mathbf{x},B)$ be a cluster
algebra. Let $(\mathbf{x}_{t},B_t)$ and $(\mathbf{x}_{t'},B_{t'})$  be two labeled seeds of
$\mathcal{A}$ and $\mathbf{i}=(i_1,i_2,\dots,i_s)$
be an $I$-sequence such that $(\mathbf{x}_{t'},B_{t'})=\mu_{\mathbf{i}}(\mathbf{x}_{t},B_t)$,
and $\sigma \in \mathcal{S}_n$.

(i) $\mathbf{i}$ is called a \textbf{$\sigma$-period }of $B_t$ if $b^{t'}_{\sigma(i)\sigma(j)}=b^t_{ij}$ holds
 for any $i,j\in I$; furthermore, if $\sigma=\mathrm{id}$, we simply call it a period of $B_{t}$.

(ii) $\mathbf{i}$ is called a \textbf{$\sigma$-period} of $(\mathbf{x}_{t},B_{t})$ if
\[b^{t'}_{\sigma(i)\sigma(j)}=b^t_{ij},\,\,\,\,\,x_{\sigma(i),t'}=x_{i,t}\]
hold for any $i,j\in I$; furthermore, if $\sigma=\mathrm{id}$, we simply call it a {\bf period} of $(\mathbf{x}_{t},B_{t})$.

\end{Definition}
If $\mathbf{i}$ is a period of $B$ (or $(\mathbf{x},B)$), we also call $\mu_{\mathbf{i}}$ a period of $B$ (or $(\mathbf{x},B)$)
without ambiguity. Let $\mathbf{i}=(i_1,i_2,\dots,i_s)$ and $\mathbf{j}=(j_1,j_2,\dots,j_p)$ be two arbitrary $I$-sequences.
Suppose that
$(\mathbf{x}_{t}, B_{t})=\mu_{\mathbf{j}}(\mathbf{x},B)$ and $\mathbf{i}$
is a period of $(\mathbf{x},B)$.
Then we have
 \[\mu_{\mathbf{j}}\mu_{\mathbf{i}}\mu_{\mathbf{j}^{-1}}(\mathbf{x}_{t},B_{t})=
\mu_{\mathbf{j}}\mu_{\mathbf{i}}\mu_{\mathbf{j}^{-1}}\mu_{\mathbf{j}}(\mathbf{x},B)=
\mu_{\mathbf{j}}\mu_{\mathbf{i}}(\mathbf{x},B)
=(\mathbf{x}_{t},B_{t}),\]
where $\mathbf{j}^{-1}= (j_p, \dots, j_2, j_1)$.
Thus, the sequence $\mathbf{j}^{-1}\mathbf{i}\mathbf{j}$ is a period of
 $(\mathbf{x}_{t},B_{t})=\mu_{\mathbf{j}}(\mathbf{x},B)$,
which induces a natural bijection between the sets of periods of two labeled seeds.
Thus we may, without loss of generality, mainly consider periods of the initial labeled seed.
For a labeled seed $(\mathbf{x},B)$, we denote by $P(\mathbf{x},B)$ the set consisting of
all mutation sequences corresponding to periods of  $(\mathbf{x},B)$.

The Extension Theorem was partially formulated
by Keller in \cite{K} and was generalized by Plamondon.
Nakanishi proves Extension and Restriction Theorem of periodicities of labeled seeds for cluster
algebras with coefficients from the universal semifield.

Let $I\subset \tilde{I}$ be two index sets, and $B=(b_{ij})_{i,j\in I}$ is the principal submatrix
of a skew-symmetrizable matrix $\tilde{B}=(\tilde{b}_{ij})_{i,j\in \tilde{I}}$ such that $B=\tilde{B}|_I$
under the restriction of the index set $I$. In this case, $B$ is called the {\bf $I$-restriction} of $\tilde{B}$
and $\tilde{B}$ is called the {\bf $\tilde{I}$-extension} of $B$.

 Also, we  call the labeled seed $((x_i)_{i\in I}, B)$
the \textbf{full subseed} of  $((x_i)_{i\in \tilde{I}}, \tilde B)$.
 In the case $\tilde B$ is skew-symmetric, the corresponding quiver $Q$ of the full subseed $((x_i)_{i\in I}, B)$ is just the full sub-quiver of the quiver $\tilde Q$ of the seed $((x_i)_{i\in \tilde I}, \tilde B)$.

\begin{Theorem}[\cite{Nak}]\label{RE} For $I\subset \tilde{I}$, let $B$ be the $I$-restriction of the skew-symmetrizable matrix $\tilde{B}$,
$\tilde{B}$ be the $\tilde{I}$-extension of $B$, and $\sigma\in S_n$.

(i)(Restriction) Assume that an $I$-sequence $\mathbf{i}=(i_1,i_2,\dots,i_s)$ is a $\sigma$-period of the
labeled seed $(\tilde{\mathbf{x}},\tilde{B})$ in $\mathcal{A}(\tilde{\mathbf{x}},\tilde{B})$, then $\mathbf{i}$
is also an $\sigma$-period of the labeled seed $(\mathbf{x},B)$ in $\mathcal{A}(\mathbf{x},B)$.

(ii)(Extension) Assume that an $I$-sequence $\mathbf{i}=(i_1,i_2,\dots,i_s)$ is a $\sigma$-period of the
labeled seed $(\mathbf{x},B)$ in $\mathcal{A}(\mathbf{x},B)$, then $\mathbf{i}$ is also an $\sigma$-period
of the labeled seed $(\tilde{\mathbf{x}},\tilde{B})$ in $\mathcal{A}(\tilde{\mathbf{x}},\tilde{B})$.

\end{Theorem}

The following example gives periods of seeds of rank two, then one can determine periods of full subseeds of rank two of any seeds of larger rank, which we will make use of in the following sections.
\begin{Example}\label{rank2}
Let $\mathcal{A}$ be a skew-symmetric cluster algebra of rank $2$ with the initial labeled seed $(\mathbf{x},B)$, and $I=[1,2]$.

(1). If $\mathbf{x}=(x_1,x_2)$, and $B=\begin{bmatrix}  0&0\\0&0  \end{bmatrix}$ is the exchange matrix whose corresponding
cluster quiver is of type $A_1\times A_1$. Then it is obvious that $\mu_1\mu_2=\mu_2\mu_1$ and periods of $(\mathbf{x},B)$
are exactly of the form $(1,2,1,2,\dots,1,2)_{2m \,\,elements}$ or $(2,1,2,1,\dots,2,1)_{2m \,\,elements}$ for some $m\in \mathbb{Z}_{>0}$.

(2). If $\mathbf{x}=(x_1,x_2)$, and $B=\begin{bmatrix}  0&1\\-1&0  \end{bmatrix}$ is the exchange matrix whose corresponding
cluster quiver is of type $A_2$. Notice that $\mathcal{A}$ is of finite type, all labeled seeds of $\mathcal{A}$
are shown in Figure 1.
There are ten various labeled seeds and five unlabeled seeds. Let $I=[1,2]$, then
the two $I$-sequences $\mathbf{i}=(1,2,1,2,1)$ and $\mathbf{j}=(2,1,2,1,2)$ are $(12)$-periods of the labeled seed
$(\mathbf{x},B)$, and the two $I$-sequences $\mathbf{p}=(1,2,1,2,1,2,1,2,1,2)$ and $\mathbf{q}=(2,1,2,1,2,1,2,1,2,1)$
are periods of $(\mathbf{x},B)$. Since $\mu_{\mathbf{p}}\mu_{\mathbf{q}}=id=\mu_{\mathbf{q}}\mu_{\mathbf{p}}$, any
period of $(\mathbf{x},B)$ is copies of $\mathbf{p}$ or $\mathbf{q}$. The two $I$-sequences $\mathbf{a}=(1,2)$ and
$\mathbf{b}=(2,1)$ are two periods of the exchange matrix $B$, and any period of $B$ is copies of $\mathbf{a}$ or $\mathbf{b}$.

On the other hand, by the Extension Theorem, for any cluster quiver which has a simple edge connecting two vertices $i$ and $j$, then actions on the labeled seed
of two mutation sequences $\mu_i\mu_j\mu_i\mu_j\mu_i$ and $\mu_j\mu_i\mu_j\mu_i\mu_j$ are both equivalent to the action
of the permutation $(ij)$ .

(3). If $\mathbf{x}=(x_1,x_2)$, and $B=\begin{bmatrix}  0&b\\-b&0  \end{bmatrix}$  with $b\geqslant 2$, then it is
well-known that any essential $I$-sequence is not a period of $(\mathbf{x},B)$, see [\cite{FZ4}, Theorem 8.8] or Theorem \ref{fintype}.

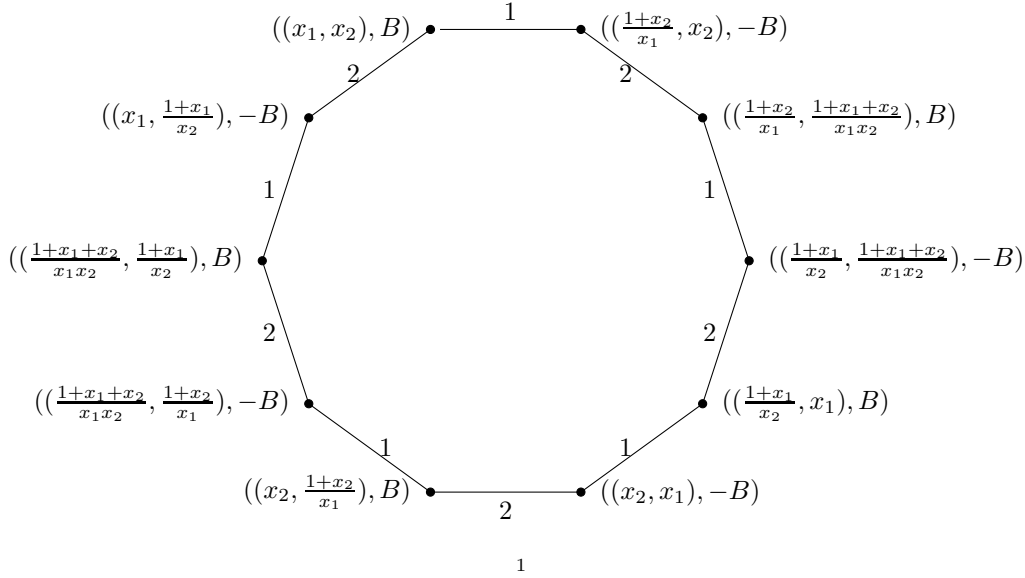
\begin{center}
 \begin{tikzpicture}

    \draw (0,0) node (1)          [label=left:$\ZK \ZK x_1\DP x_2\YK \DP B \YK$] {}
        -- node [left] {$2$}++(216:2.0cm) node (2) [label=left:$\ZK \ZK x_1\DP \frac{ 1\PL x_1}{ x_2} \YK \DP -B \YK$] {}
        -- node [left] {$1$}++(252:2.0cm) node (3) [label=left:$\ZK \ZK\frac{1\PL x_1\PL x_2}{x_1x_2}\DP \frac{ 1\PL x_1}{ x_2} \YK \DP B \YK$] {}
        -- node [left] {$2$}++(288:2.0cm) node (4) [label=left:$\ZK \ZK\frac{1\PL x_1\PL x_2}{x_1x_2}\DP \frac{ 1\PL x_2}{ x_1} \YK \DP -B \YK$] {}
        -- node [right] {$1$}++(324:2.0cm) node (5) [label=left:$\ZK \ZK x_2\DP \frac{ 1\PL x_2}{ x_1} \YK \DP B \YK$] {}
        -- node [below] {$2$}++(0:2.0cm) node (6)  [label=right:$\ZK \ZK x_2\DP x_1 \YK \DP -B \YK$] {}
        -- node [left] {$1$}++(36:2.0cm) node (7) [label=right:$\ZK \ZK \frac{1\PL x_1}{x_2}\DP x_1 \YK \DP B \YK$] {}
        -- node [left] {$2$}++(72:2.0cm) node (8) [label=right:$\ZK \ZK \frac{1\PL x_1}{x_2}\DP \frac{1\PL x_1\PL x_2}{x_1x_2} \YK \DP -B \YK$] {}
        -- node [left] {$1$}++(108:2.0cm) node (9) [label=right:$\ZK \ZK \frac{1\PL x_2}{x_1}\DP \frac{1\PL x_1\PL x_2}{x_1x_2} \YK \DP B \YK$] {}
        -- node [left] {$2$}++(144:2.0cm) node (10) [label=right:$\ZK \ZK \frac{1\PL x_2}{x_1}\DP x_2 \YK \DP -B \YK$] {}
        -- node [above] {$1$}(1);
\foreach \i in {1,2,...,10}
	{
	  \draw[fill=black] (\i) circle (0.15em);	
	}

\end{tikzpicture}
\end{center}
\begin{center}
 \begin{tikzpicture}
node [label=90:$\mathrm{Figure}\,\, \mathrm{1}$]
\end{tikzpicture}
\end{center}

\end{Example}

\subsection{Mutation-periodic groups}\label{main}
In the sequel, we define mutation-periodic groups for exchange matrices and labeled seeds. For
 a labeled seed $(\mathbf{x},B)$, let $G(B)$ be the opposite groups of the subgroup of the mutation group consisting of
 all mutation sequences which keep $B$ invariant, and $H(\mathbf{x},B)$ be the opposite group of the subgroup of
 the mutation group consisting of all mutation sequences which keep $(\mathbf{x},B)$ invariant i.e.,
\[G(B)=<\mu_{\mathbf{i}},\mathbf{i}\; \mbox{is a period of}\; B>^{op};  \]
\[H(\mathbf{x},B)=<\mu_{\mathbf{i}},\mathbf{i}\; \mbox{is a period of}\; (\mathbf{x},B)>^{op}. \]

Since a (labeled) seed is determined by its cluster, we may denote $H(\mathbf{x},B)$ by $H(\mathbf{x})$ for simplicity.
Let $\mu_{\mathbf{i}}=\mu_{i_s}\dots\mu_{i_2}\mu_{i_1}$ and $\mu_{\mathbf{j}}= \mu_{j_p}\dots\mu_{j_2}\mu_{j_1}$
be two mutation sequences in $G(B)(\mbox{resp.}\; H(\mathbf{x}) )$, the multiplication of $G(B)(\mbox{resp.} \; H(\mathbf{x}))$ is given by
\[\mu_{\mathbf{i}}\circ \mu_{\mathbf{j}} = \mu_{j_p}\dots\mu_{j_2}\mu_{j_1}\mu_{i_s}\dots\mu_{i_2}\mu_{i_1}.\]
Note that $id \in G(B)$ and for any $\mu_{\mathbf{i}}=\mu_{i_s}\dots\mu_{i_2}\mu_{i_1}\in G(B)$, it is obvious that $\mu_{i_1}\mu_{i_2}\dots\mu_{i_s}\in G(B)$
and \[\mu_{i_s}\dots\mu_{i_2}\mu_{i_1}\circ \mu_{i_1}\mu_{i_2}\dots\mu_{i_s}=id=\mu_{i_1}\mu_{i_2}\dots\mu_{i_s}\circ \mu_{i_s}\dots\mu_{i_2}\mu_{i_1}.\]
Thus for each $\mu_{\mathbf{i}}=\mu_{i_s}\dots\mu_{i_2}\mu_{i_1}\in G(B)$, $\mu_{\mathbf{i}}^{-1}= \mu_{i_1}\mu_{i_2}\dots\mu_{i_s}=\mu_{\mathbf{i}^{-1}}$.

\begin{Lemma} Let $I\subset \tilde{I}$, and let $B$ be the $I$-restriction of the skew-symmetric matrix $\tilde{B}$,
then $H(\mathbf{x},B)$ is a subgroup of $H(\tilde{\mathbf{x}},\tilde{B})$.
\end{Lemma}
\begin{proof} This follows from Theorem \ref{RE} immediately.
\end{proof}

\begin{Lemma}\label{norm1}
For a labeled seed $(\mathbf{x},B)$ of a cluster algebra $\mathcal{A}$,
the mutation-periodic group $H(\mathbf{x})$ of $(\mathbf{x},B)$ is a normal subgroup of
$G(B)$.
\end{Lemma}
\begin{proof} Let $\mu_{\mathbf{j}}= \mu_{j_p}\dots\mu_{j_2}\mu_{j_1}$ and $\mu_{\mathbf{i}}=\mu_{i_s}\dots\mu_{i_2}\mu_{i_1}$
be two mutation sequences in $G(B)$ and $H(\mathbf{x},B)$, respectively. Since $\mu_{\mathbf{i}}(\mathbf{x},B)=(\mathbf{x},B)$
implies $\mu_{\mathbf{i}}(\mathbf{x'},B)=(\mathbf{x'},B)$ for an arbitrary ordered set $\mathbf{x}'$ of $n$ independent variables.
Without loss of generality, assume that $\mu_{\mathbf{j}}(\mathbf{x},B)=(\mathbf{x''},B)$, then $\mu_{\mathbf{j}}^{-1}(\mathbf{x''},B)=(\mathbf{x},B)$.
Thus we have
\[ \mu_{\mathbf{j}}^{-1}\mu_{\mathbf{i}} \mu_{\mathbf{j}}(\mathbf{x},B)=\mu_{\mathbf{j}}^{-1}\mu_{\mathbf{i}}(\mathbf{x''},B)=\mu_{\mathbf{j}}^{-1}(\mathbf{x''},B)=(\mathbf{x},B).\]

\end{proof}

Now we consider relations between mutation-periodic groups and strict direct cluster automorphism groups. Indeed,
For any labeled seed $(\mathbf{x},B)$ of a cluster algebra $\mathcal{A}$, $\mathrm{SAut}^+(\mathbf{A})$ is isomorphic
to the quotient group of $G(B)$ modulo the normal subgroup $H(\mathbf{x})=H(\mathbf{x},B)$. In fact, the following
theorem holds.

\begin{Theorem}\label{exa1}
Let $(\mathbf{x},B)$ be a labeled seed of a cluster algebra $\mathcal{A}$, then there is an exact sequence
\[ 1\rightarrow H(\mathbf{x})\rightarrow G(B)\xrightarrow{\varphi} \mathrm{SAut}^+(\mathcal{A})\rightarrow 1.          \]
\end{Theorem}
\begin{proof}   Let us define a direct cluster automorphism $f_{\mathbf{i}}:\mathcal{A}\rightarrow \mathcal{A}$ associated to $\mu_{\mathbf{i}}=\mu_{i_s}\dots\mu_{i_2}\mu_{i_1}\in G(B)$  as follows. Suppose that $\mu_{\mathbf{i}}(\mathbf{x},B)=(\mathbf{z},B)$.
Let $f_{\mathbf{i}}: \mathcal{F} \rightarrow \mathcal{F}$  be
 the homomorphism of $\mathcal{F}$ such that $f_{\mathbf{i}}(\mathbf{x})=\mathbf{z}$ as ordered sets, where $\mathcal{F} = \mathbb{Q}(\mathbf{x}) =\mathbb{Q}(\mathbf{z}) $ is the ambient field of $\mathcal{A}$. More precisely, $f_{\mathbf{i}}$ is given by \[\frac{p(x_1,x_2,\dots,x_n)}{q(x_1,x_2,\dots,x_n)}\mapsto \frac{p(z_1,z_2,\dots,z_n)}{q(z_1,z_2,\dots,z_n)}\]
 for all polynomials $p,q\neq 0$.
  It is clear that $f_{\mathbf{i}}$ is an automorphism of $\mathcal{F}$ since every cluster is a transcendental basis of $\mathcal{F}$. We will have that $f_{\mathbf{i}}(\mu_k(\mathbf{x}))=\mu_k(f_{\mathbf{i}}(\mathbf{x}))$ for any $k\in [1,n]$. Indeed for any $k\in [1,n]$, assume that $\mu_k(\mathbf{x})=(x_1,\dots,x'_k,\dots,x_n)$  and $\mu_k(f(\mathbf{x})) = \mu_k(\mathbf{z})=(z_1,\dots,z'_k,\dots,z_n)$. Then we have
\[f_{\mathbf{i}}(x'_k) = f_{\mathbf{i}}(\frac{\prod\limits_{i=1}^n x_i^{[b_{ik}]_+}  + \prod\limits_{i=1}^n x_i^{[-b_{ik}]_+}}{x_k}) =   \frac{\prod\limits_{i=1}^n (x'_i)^{[b_{ik}]_+}  + \prod\limits_{i=1}^n (x'_i)^{[-b_{ik}]_+}}{x'_k}  =z'_k.    \]
 Thus, $f_{\mathbf{i}}$ is an automorphism of $\mathcal{F}$ satisfying that $f_{\mathbf{i}}(\mu_k(\mathbf{x}))=\mu_k(f(\mathbf{x}))$ for any $k\in [1,n]$.
 By [Lemma 3.5, \cite{CLLP}], then $f_{\mathbf{i}}$ is a cluster
automorphism of $\mathcal{A}$, and it is clear that $f_{\mathbf{i}}$ is direct and strict by definitions.

Define a map $\varphi: G(B)\rightarrow \mathrm{SAut}^+(\mathcal{A})$ by $\varphi(\mu_{\mathbf{i}})=f_{\mathbf{i}}$.
Obviously, $\varphi$ is well-defined.  For any $\mu_{\mathbf{i}}, \mu_{\mathbf{j}}\in G(B)$, we have that
\begin{equation}\label{eq1}
\varphi(\mu_{\mathbf{i}}\circ \mu_{\mathbf{j}})(\mathbf{x})=\varphi(\mu_{\mathbf{j}}\mu_{\mathbf{i}})(\mathbf{x})=
f_{\mathbf{ji}}(\mathbf{x})=\mu_{\mathbf{j}}\mu_{\mathbf{i}}(\mathbf{x})=\mu_{\mathbf{j}}(f_{\mathbf{i}}(\mathbf{x}))=
f_{\mathbf{i}}(\mu_{\mathbf{j}}(\mathbf{x}))=f_{\mathbf{i}}(f_{\mathbf{j}}(\mathbf{x}))=(\varphi(\mu_{\mathbf{i}})\varphi(\mu_{\mathbf{j}}))(\mathbf{x})
\end{equation}
 for any cluster, where the fifth equality in (\ref{eq1}) is due to the compatibility of strict direct cluster automorphisms and mutations.  Following this, $\varphi$ is a homomorphism of groups.

We show that $\varphi$ is surjective. Indeed, for any strict cluster automorphism
$f:\mathcal{A}\rightarrow \mathcal{A}$, it maps every labeled seed $\mathcal{A}$ to
a labeled seed of $\mathcal{A}$ with the same exchange matrices. For the given labeled
seed $(\mathbf{x},B)$, $f(\mathbf{x})$ is a cluster with the exchange matrix $B$, i.e.,
there exists a mutation sequence $\mu_{\mathbf{i}}=\mu_{i_s}\dots\mu_{i_2}\mu_{i_1}\in G(B)$
such that $\mu_{\mathbf{i}}(\mathbf{x},B)=(f(\mathbf{x}),B)$ as labeled seeds.
By the construction of the direct cluster automorphism $f_{\mathbf{i}}$ associated to $\mu_{\mathbf{i}}$,
we have that $f_{\mathbf{i}}(\mathbf{x}) = f(\mathbf{x})$ as ordered sets, i.e., $f_{\mathbf{i}}(x_j) = f(x_j)$ for any $j\in[1,n]$. Then we obtain that  $f=f_{\mathbf{i}}=\varphi(\mu_{\mathbf{i}})$ on $\mathcal F$ via extending the cluster automorphisms  uniquely  to an automorphism of the ambient field $\mathcal{F}$ since  $\mathcal{F}$ is generated freely by $\mathbf{x}$.

It remains to prove $\mathrm{ker}\varphi=H(\mathbf{x})$. Since a strict cluster automorphism in uniquely determined by
its value on an arbitrary labeled cluster, $\varphi(\mu_{\mathbf{i}})=id$ if and only if $\varphi(\mu_{\mathbf{i}})
(\mathbf{x})=\mathbf{x}$ as labeled clusters, which is equivalent to say  $\mu_{\mathbf{i}}(\mathbf{x})=\mathbf{x}$,
i.e., $\mathbf{i}$ is a period of $(\mathbf{x},B)$.
\end{proof}

\begin{Remark}
For a skew-symmetric cluster algebra,
King and Pressland also proved a similar result for $\mathrm{Aut}^+(\mathcal{A})$
 by also considering the action of $S_n$.

\end{Remark}

\subsection{$\mathrm{SAut}^+(\mathcal{A})$ as normal groups via permutation groups}

\begin{Lemma}\label{PL}
 Let $(\mathbf{x},B)$ be the initial labeled seed of $\mathcal{A}$.

(1) For $\sigma \in S_n$, there exists an $I$-sequence $\mathbf{i}$ as
 a $\sigma$-period of $B$ (respectively, $(\mathbf{x},B)$) if and only if there exists an  $I$-sequence $\mathbf{j}$ as
 a $\sigma$-period of $B'$ (respectively, $(\mathbf{x}',B')$) for any $B'$(respectively, $(\mathbf{x}',B')$) mutation equivalent to $B$ (respectively, $(\mathbf{x},B)$).

(2) Let $L_n$ and $P_n$ be two subsets of $S_n$ defined as follows.
\[L_n=\{\sigma\in S_n|\mbox{there exists an $I$-sequence $\mathbf{i}$ such that $\mu_\mathbf{i}(B)=B^{\sigma}$\}},\]
\[P_n=\{\tau\in  S_n|\mbox{there exists an $I$-sequence $\mathbf{j}$ such that $\mu_\mathbf{j}(\mathbf{x},B)=(\mathbf{x}^{\tau},B^{\tau})$}\}.\]
Then $L_n$ and $P_n$ are subgroups of $S_n$ and are independent with the choice of labeled seeds in the same mutation class.

(3) $P_n$ is a normal subgroup of the permutation group  $L_n$.
\end{Lemma}

\begin{proof}
It is easy to verify that for any $k\in [1,n], \sigma,\tau \in S_n$, we have
\begin{equation}\label{wk}
 \mu_k(\mathbf{x}^{\sigma},B^{\sigma})=\sigma(\mu_{\sigma(k)}(\mathbf{x},B)),\,\,\, (\sigma\tau)(\mathbf{x},B)=\tau(\sigma(\mathbf{x},B)).
\end{equation}
The second equality follows from that
\[\tau(\sigma(x_1, x_2, \dots, x_n))=\tau(x_{\sigma(1)},x_{\sigma(2)},\dots, x_{\sigma(n)})
=(x_{\sigma(\tau(1))},x_{\sigma(\tau(2))},\dots, x_{\sigma(\tau(n))}).\]

For (1), if $\mu_{i_k}\dots \mu_{i_2}\mu_{i_1}(\mathbf{x},B)= (\mathbf{x},B)^{\sigma}$ and $(\mathbf{x}',B')=\mu_{p_s}\dots \mu_{p_2}\mu_{p_1}(\mathbf{x},B)$, then
by (\ref{wk}),
\[\begin{split} (\mathbf{x}',B')^{\sigma} &=(\mu_{p_s}\dots \mu_{p_2}\mu_{p_1}(\mathbf{x},B))^{\sigma}\\
&=  \mu_{\sigma^{-1}(p_s)}\dots \mu_{\sigma^{-1}(p_2)}\mu_{\sigma^{-1}(p_1)}((\mathbf{x},B)^{\sigma})\\
&=  \mu_{\sigma^{-1}(p_s)}\dots \mu_{\sigma^{-1}(p_2)}\mu_{\sigma^{-1}(p_1)}\mu_{i_k}\dots \mu_{i_2}\mu_{i_1}(\mathbf{x},B)\\
&= \mu_{\sigma^{-1}(p_s)}\dots \mu_{\sigma^{-1}(p_2)}\mu_{\sigma^{-1}(p_1)}\mu_{i_k}\dots \mu_{i_2}\mu_{i_1}\mu_{p_1}\dots\mu_{p_s}(\mathbf{x}',B').
\end{split} \]

For (2), it follows easily from (\ref{wk}) and (1).

For (3),  let $\sigma\in L_n, \tau\in P_n$,
$\mu_{\mathbf{i}}=\mu_{i_s}\dots\mu_{i_2}\mu_{i_1}$ and $\mu_{\mathbf{j}}= \mu_{j_p}\dots\mu_{j_2}\mu_{j_1}$
be two mutation sequences such that $\mu_\mathbf{i}(\mathbf{x},B)=(\mathbf{x}', B^{\sigma})$ and
$\mu_{\mathbf{j}}(\mathbf{x},B)=(\mathbf{x}^{\tau},B^{\tau})$, respectively.
Then we have
 \[\begin{split} (\sigma^{-1}\tau\sigma)(\mathbf{x},B)&=(\sigma^{-1}\tau\sigma)\mu_{i_1}\mu_{i_2}\dots\mu_{i_s}(\mathbf{x}', B^{\sigma})\\
&=\mu_{\sigma^{-1}\tau^{-1}\sigma(i_1)}\mu_{\sigma^{-1}\tau^{-1}\sigma(i_2)}\dots\mu_{\sigma^{-1}\tau^{-1}\sigma(i_s)}((\tau\sigma)(\mathbf{x}'^{\sigma^{-1}}, B))\\
&=\mu_{\sigma^{-1}\tau^{-1}\sigma(i_1)}\mu_{\sigma^{-1}\tau^{-1}\sigma(i_2)}\dots\mu_{\sigma^{-1}\tau^{-1}\sigma(i_s)}(\sigma(\tau(\mathbf{x}'^{\sigma^{-1}}, B)))\\
&=\mu_{\sigma^{-1}\tau^{-1}\sigma(i_1)}\mu_{\sigma^{-1}\tau^{-1}\sigma(i_2)}\dots\mu_{\sigma^{-1}\tau^{-1}\sigma(i_s)}
(\sigma\mu_{j_p}\dots\mu_{j_2}\mu_{j_1}(\mathbf{x}'^{\sigma^{-1}}, B))
\\&=\mu_{\sigma^{-1}\tau^{-1}\sigma(i_1)}\mu_{\sigma^{-1}\tau^{-1}\sigma(i_2)}\dots\mu_{\sigma^{-1}\tau^{-1}\sigma(i_s)}
\mu_{\sigma^{-1}(j_p)}\dots\mu_{\sigma^{-1}(j_2)}\mu_{\sigma^{-1}(j_1)}(\mathbf{x}', B^{\sigma})
\\&=\mu_{\sigma^{-1}\tau^{-1}\sigma(i_1)}\mu_{\sigma^{-1}\tau^{-1}\sigma(i_2)}\dots\mu_{\sigma^{-1}\tau^{-1}\sigma(i_s)}
\mu_{\sigma^{-1}(j_p)}\dots\mu_{\sigma^{-1}(j_2)}\mu_{\sigma^{-1}(j_1)}\mu_{i_s}\dots\mu_{i_2}\mu_{i_1}(\mathbf{x},B), \end{split}
\]
which implies that $\sigma^{-1}\tau\sigma \in P_n$.
\end{proof}

The permutation groups $L_n$ and $P_n$ are called
\textbf{permutation-periodic groups} of exchange matrices and labeled seeds respectively.

\begin{Theorem}\label{thm2} Let $\mathcal{A}=\mathcal{A}(\mathbf{x},B)$ be a cluster algebra, and let $\mathrm{SAut}^+(\mathcal{A})$ and
$\mathrm{Aut}^+(\mathcal{A})$ be its strict direct cluster automorphism group and direct cluster automorphism group, respectively.
Suppose that $L_n$ and $P_n$ are defined as above. Then

(i) $\mathrm{SAut}^+(\mathcal{A})$ is a normal subgroup of $\mathrm{Aut}^+(\mathcal{A})$;

(ii) there is an exact sequence of groups
 \[ 1\rightarrow \mathrm{SAut}^+(\mathcal{A}) \rightarrow \mathrm{Aut}^+(\mathcal{A}) \xrightarrow{\phi} L_n/P_n \rightarrow 1.\]
\end{Theorem}
\begin{proof}
(i)
Assume $f\in S\mathrm{Aut}^+(\mathcal{A}), g\in \mathrm{Aut}^+(\mathcal{A})$, we prove that $gfg^{-1}\in \mathrm{SAut}^+(\mathcal{A})$.

Indeed, let $x'=g(x)$ and $\mu_{\mathbf{i}}(\mathbf{x}')=f(\mathbf{x}')$. Then
\[gfg^{-1}(x)=gf(x')=g(\mu_{\mathbf{i}}(\mathbf{x}')) = \mu_{\mathbf{i}}g(\mathbf{x}') = \mu_{\mathbf{i}}(\mathbf{x}),\]
and so $\mathrm{SAut}^+(A)$ is a normal subgroup as claimed.

(ii) As noted in Remark \ref{rm2}, for a direct cluster automorphism $f:\mathcal{A}\rightarrow \mathcal{A}$, there exists a mutation
sequence $\mu_{\mathbf{i}}$ and a permutation $\sigma\in S_n$ such that $f(\mathbf{x})=\sigma(\mu_{\mathbf{i}}(\mathbf{x}))$
and $B(\mu_{\mathbf{i}}(\mathbf{x}))=B^{\sigma^{-1}}$. Then we define a map $\phi$ mapping $f$ to $\sigma P_n=P_n\sigma$.
Assume there exists another sequence $\mathbf{j}=(j_1,j_2,\dots,j_p)$ and a permutation $\eta\in S_n$ such that
$f(\mathbf{x})=\eta(\mu_{\mathbf{j}}(\mathbf{x}))$
and $B(\mu_{\mathbf{j}}(\mathbf{x}))=\mu_{\mathbf{i}}(B)=B^{\eta^{-1}}$. Then we have
\[\sigma(\mu_{\mathbf{i}}(\mathbf{x}))=f(\mathbf{x})=\eta(\mu_{\mathbf{j}}(\mathbf{x})),\]
and thus
\[\mu_{\eta\sigma^{-1}(i_1)}\dots\mu_{\eta\sigma^{-1}(i_s)}\mu_{j_p}\dots\mu_{j_1}(\mathbf{x},B)
=(\sigma\eta^{-1})(\mathbf{x},B).\]
Thus $\sigma\eta^{-1}\in P_n$ and $\sigma P_n= P_n\sigma=P_n\eta=\eta P_n$.

Let $f$ and $g$ be two direct cluster automorphisms, then there exist two sequences $\mathbf{i}=(i_1,i_2,\dots,i_s)$,
$\mathbf{j}=(j_1,j_2,\dots, j_p)$ of and two permutations $\sigma,\tau\in S_n$ such that
\[ \sigma(\mu_{\mathbf{i}}(\mathbf{x}))=f(\mathbf{x}),\, B(\mu_{\mathbf{i}}(\mathbf{x}))=B^{\sigma^{-1}};\,\,
\tau(\mu_{\mathbf{j}}(\mathbf{x}))=g(\mathbf{x}),\, B(\mu_{\mathbf{j}}(\mathbf{x}))=B^{\tau^{-1}}.\]
Then we have \[ gf(\mathbf{x})=g(\sigma\mu_{\mathbf{i}}(\mathbf{x}))=\sigma\mu_{\mathbf{i}}g(\mathbf{x})=
\sigma\mu_{\mathbf{i}}\tau\mu_{\mathbf{j}}(\mathbf{x})=\sigma(\tau\mu_{\mathbf{\tau(i)}}\mu_{\mathbf{j}}(\mathbf{x}))
=(\tau\sigma)\mu_{\mathbf{\tau(i)}}\mu_{\mathbf{j}}(\mathbf{x}),\]
and
\[\mu_{\mathbf{\tau(i)}}\mu_{\mathbf{j}}(B)=\mu_{\mathbf{\tau(i)}}(\tau^{-1}(B))=\tau^{-1}(\mu_{\mathbf{i}}(B))
=\tau^{-1}(\sigma^{-1}(B))=(\sigma^{-1}\tau^{-1})(B)=B^{(\tau\sigma)^{-1}}.\]
Thus we have \[\phi(gf)=(\tau\sigma)P_n=(\tau P_n)(\sigma P_n)=\phi(g)\phi(f).\]

Hence, $\phi$ is well-defined as a homomorphism of groups.
 It is surjective, since for any $\sigma\in L_n$, we have $\sigma^{-1}\in L_n$
and hence there is a sequence  $\mathbf{i}$ such that $B(\mu_{\mathbf{i}}(\mathbf{x}))=B^{\sigma^{-1}}$.
Let us define a cluster automorphism $f: \mathcal{A}\rightarrow \mathcal{A}$ as follows.
For the labeled seed $(\mathbf{x}, B)$, assume that $\mu_{\mathbf{i}}(\mathbf{x}, B) = (\mathbf{z}, B^{\sigma^{-1}})$ and consider the seed $(\mathbf{z}^{\sigma}, B)$ of $\mathcal{A}$. Let $f: \mathcal{F} \rightarrow \mathcal{F}$  be
 the homomorphism of $\mathcal{F}$ such that $f(\mathbf{x})=\mathbf{z}^{\sigma}$ as ordered sets, i.e., $f(x_j)=z_{\sigma(j)}$ for each $j\in [1,n]$.
It is clear that $f$ is an automorphism of $\mathcal{F}$ and it is easy to check that $f(\mu_{x_k}(\mathbf{x})) = \mu_{f(x_k)}(\mathbf{z}^{\sigma})$ for every
$k\in [1,n]$. By [Lemma 3.5, \cite{CLLP}], we know that $f$ is a direct cluster automorphism such that $f(\mathbf{x})=\mathbf{z}^{\sigma}$.
By the construction of the map  $\phi$, we know that $\phi(f)=\sigma P_n$ and thus $\phi$ is surjective.  It remains to prove that $\mathrm{ker}(\phi)=\mathrm{SAut}^+(\mathcal{A})$. It is clear that $\mathrm{ker}(\phi)\supset \mathrm{SAut}^+(\mathcal{A})$.
Conversely, for $f\in \mathrm{ker}\phi$, there are $\sigma$ and $\mu_{\mathbf{i}}$ such that $\sigma(\mu_{\mathbf{i}}(\mathbf{x}))=f(\mathbf{x}),\, B(\mu_{\mathbf{i}}(\mathbf{x}))=B^{\sigma^{-1}}$. If $\sigma\in P_n$, then
there is a sequence $\mathbf{j}$ such that $\mu_{\mathbf{j}}(\mathbf{x},B)=(\mathbf{x}^{\sigma},B^{\sigma})$. Then we have
\[f(\mathbf{x})=\sigma(\mu_{\mathbf{i}}(\mathbf{x}))=\mu_{\sigma^{-1}(\mathbf{i})}\sigma(\mathbf{x})=\mu_{\sigma^{-1}(\mathbf{i})}\mu_{\mathbf{j}}(\mathbf{x}).\]
Thus similar to the proof in Theorem \ref{exa1}, we can prove that $f=f_{\mathbf{j},\sigma^{-1}(\mathbf{i})}\in \mathrm{SAut}^+(\mathcal{A})$.
\end{proof}

\begin{Corollary}\label{finitegp}
Let $\mathcal{A}=\mathcal{A}(\mathbf{x},B)$ be a cluster algebra. Then $\mathrm{SAut}^+(\mathcal{A})$ is a finite group if and only if
$\mathrm{Aut}(\mathcal{A})$ is a finite group.
\end{Corollary}
\begin{proof}
It follows from Theorem \ref{thm2} that $\mathrm{SAut}^+(\mathcal{A})$ is a finite group if and only if
$\mathrm{Aut}^+(\mathcal{A})$ is a finite group. The latter one is finite if and only if $\mathrm{Aut}(\mathcal{A})$ is finite
by Corollary \ref{cor1}.
\end{proof}

From Theorem \ref{thm2}, we know that $\mathrm{SAut}^+(\mathcal{A})=\mathrm{Aut}^+(\mathcal{A})$ is equivalent to $P_n=L_n$.
In general, a strict direct cluster automorphism group $\mathrm{SAut}^+(\mathcal{A})$ is a proper subgroup of $\mathrm{Aut}^+(\mathcal{A})$.
For example, the exchange matrix $B$ is the matrix of a Kronecker quiver, then $L_2 = S_2$, however $P_2 = 1$.
In the rest of this section, we give a sufficient condition such that $\mathrm{SAut}^+(\mathcal{A})=\mathrm{Aut}^+(\mathcal{A})$ holds. Actually,
we obtain a number of exchange matrices such that any permutation of the initial labeled seed is mutation equivalent to it as
labeled seeds.

For a skew-symmetric matrix $B=(b_{ij})_{n\times n}$, let $v(B)=\mathrm{max}\{b_{ij}|i,j\in [1,n]\}$ be the maximal entry of $B$,
and let $m(B)= \mathrm{inf}\{v(B')|B'\, \mbox{is mutation equivalent to}\, B \}$ be the infimum of all maximal elements of exchange matrices
occurring in the mutation class of $B$.
\begin{Lemma}[\cite{D}]\label{conn} For any connected quiver $Q$, its vertices can always be enumerated, say as $q_1, q_2,\dots,q_n$, so that the full subquiver $Q[q_1, q_2,\dots,q_i]$ of $Q$ determined
by $\{q_1, q_2,\dots,q_i\}$ is connected for every $i\in [1,n]$.
\end{Lemma}

\begin{Proposition}\label{m1}
If $B$ is a skew-symmetric matrix with $m(B)=1$, then $P_n=S_n$ holds, that is,
 for any $\sigma\in S_n$, there exists an $I$-sequence $\mathbf{j}$ such that
 $(\mathbf{x},B)^{\sigma}=\mu_{\mathbf{j}}(\mathbf{x},B)$.
 Moreover, $\mathrm{SAut}^+(\mathcal{A})=\mathrm{Aut}^+(\mathcal{A})$ holds.
\end{Proposition}
\begin{proof} By Lemma \ref{PL}(1), without loss of generality, we may assume that $v(B)=1$. Since $B$ is indecomposable,
 the corresponding quiver $Q$ is connected. For any $\sigma\in S_n$, we want to find an $I$-sequence $\mathbf{j}$ such that
 $(\mathbf{x},B)^{\sigma}=\mu_{\mathbf{j}}(\mathbf{x},B)$  in at most $n-1$ steps.

The idea of the   following proof is that if two vertices $i$ and $j$ of a quiver $Q$ are connected by only one edge, by Theorem \ref{RE} and Example \ref{rank2}, the sequence $(i,j,i,j,i)$ is a $\sigma=(ij)$-period of $(\mathbf{x},Q)$. By applying mutation sequence $\mu_i\mu_j\mu_i\mu_j\mu_i$ to the seed $(\mathbf{x},Q)$, the underlying graph of the new quiver does not changed, however we permute cluster variables $x_i$ and $x_j$. We will repeat this operation to obtain the desired result.

  {\bf Step 1}. By Lemma \ref{conn}, we can take $i_n\in Q_0$ such that the full subquiver $Q^{(1)}$ of $Q$ determined by the subset of vertices
  $Q_0\backslash \{v_n\}$ is connected, where $v_n$ is the vertex of $Q$ corresponding to $x_n$. In this case, $v_n=i_n$.
  If $\sigma^{-1}(i_n)=v_n$, then $x_{i_n}$ lies in the $\sigma^{-1}(i_n)$-th position and proceed the next step.
   Otherwise $\sigma^{-1}(i_n)\neq v_n$, since $Q$ is connect, there is a sequence  $\{w_0,w_1,\dots,w_m\}$ such that
  $w_0=\sigma^{-1}(i_n), w_m=v_n$, and $(w_k, w_{k+1})$ is an edge in $Q$ for $0\leq k<m$, i.e., $|b_{w_kw_{k+1}}|=1$.
  By applying the following mutation sequence
  \[\mu^{(n)}=(\mu_{w_0}\mu_{w_m}\mu_{w_0}\mu_{w_m}\mu_{w_0})\dots (\mu_{w_0}\mu_{w_2}\mu_{w_0}\mu_{w_2}\mu_{w_0})(\mu_{w_0}\mu_{w_1}\mu_{w_0}\mu_{w_1}\mu_{w_0})\]
  to the initial labeled seed $(\mathbf{x},Q)$, we get a new labeled seed $(\mathbf{x}',Q')$ such that the cluster variable
  $x_{i_n}$ lies in the $w_0$-th position, i.e., the $\sigma^{-1}(i_n)$-th position, in $\mathbf{x}'$.

  Note that $Q$ and $Q'$ share the same underlying graph and the full subquiver $\hat{Q}'$ of $Q'$ determined by
  $Q'_0\backslash \{\sigma^{-1}(i_n),\,\,\text{i.e., the vertex corresponding to}\,\, x_{i_n}\,\, \text{in}\,\, Q' \}$ is connected
  since $\hat{Q}'$ and $Q^{(1)}$ also have the same underlying graph and $Q^{(1)}$ is connected.

 {\bf Step 2}. By Lemma \ref{conn} and $\hat{Q}'$ is connneted, we can take $i_{n-1}\in[1,n]\backslash \{i_n\}$ such that the full subquiver $Q^{(2)}$ of $Q'$ determined by
 $$Q'_0\backslash \{\text{the vertices corresponding to}\,\, x_{i_{n-1}}\,\, \text{and}\,\, x_{i_n}\,\, \text{in}\,\, Q' \}$$ is connected.
 Assume that the vertex corresponding to $x_{i_{n-1}}$ in $Q'$ is $v_{n-1}$.
 If $\sigma^{-1}(i_{n-1})=v_{n-1}$, then $x_{i_{n-1}}$ is in the
 $\sigma^{-1}(i_{n-1})$-th position and proceed the next step. Otherwise, note that $v_{n-1}, \sigma^{-1}(i_{n-1})\in \hat{Q}'_0$
 and $\hat{Q}'$ is connected, we may take a sequence $\{s_0,s_1,\dots,s_l\}$ of vertices in $\hat{Q}'_0$ such that
 $s_0=\sigma^{-1}(i_{n-1}), s_l=v_{n-1}$, and $(s_k, s_{k+1})$ is an edge in $\hat{Q}'$ for $0\leq k<l$. Apply the following
 mutation sequence
\[\mu^{(n-1)}=(\mu_{s_0}\mu_{s_l}\mu_{s_0}\mu_{s_l}\mu_{s_0})\dots (\mu_{s_0}\mu_{s_2}\mu_{s_0}\mu_{s_2}\mu_{s_0})(\mu_{s_0}\mu_{s_1}\mu_{s_0}\mu_{s_1}\mu_{s_0})\]
 to the labeled seed $(\mathbf{x'},Q')$, then we get a labeled seed $(\mathbf{x''},Q'')$ in which $x_{i_{n-1}}$ is in the
 $\sigma^{-1}(i_{n-1})$-th position and $x_{i_n}$ is in the $\sigma^{-1}(i_n)$-th position in $\mathbf{x}''$.

 Note that $Q''$ and $Q'$ share the same underlying graph and the full subquiver $\hat{Q}''$ of $Q''$ determined by
  $Q''_0\backslash \{\sigma^{-1}(i_{n-1}),\sigma^{-1}(i_n),\text{i.e., the vertices corresponding to}\,\,  x_{i_{n-1}}\,\, \text{and}\,\, x_{i_n} \,\, \text{in}\,\, Q''\}$ is also connected since $\hat{Q}''$ and $Q^{(2)}$ have the same underlying graph and $Q^{(2)}$ is connected.

 Repeat the above steps more at most $n-3$ times, we will obtain a labeled seed such that for all $j\in [1,n]$, $x_j$ lies in the $\sigma^{-1}(j)$-th
 position, which is actually the labeled seed $(\mathbf{x},B)^{\sigma}$.
\end{proof}

\begin{Example}
Let $(\mathbf{x},Q)$ be a labeled seed where $\mathbf{x}=(x_1,x_2,x_3,x_4)$ and $Q$ is given as follows.
\[\begin{tikzpicture}[scale=1.3]
\node at (0,0) (11) {$1$};
\node at (1,0) (21) {$2$};
\node at (2,0) (31) {$3$};
\node at (3,0) (41) {$4$};
\path[-angle 90]
	(11) edge (21)
	(21) edge (31)
	(31) edge (41);
\end{tikzpicture}\]
Assume that $\sigma=(14)(23)\in S_4$. We construct a mutation sequence $\mu_{\mathbf{j}}$ such that $\mu_{\mathbf{j}}(\mathbf{x},Q)=(\mathbf{x},Q)^{\sigma}$
using the method in the proof of Proposition \ref{m1}.

 Step 1. Take $i_4=4$, then $v_4=4$ and $\sigma^{-1}(i_4)=1$. Let $w_0=1, w_1=2, w_2=3$ and $w_3=4$, thus \[\mu^{(4)}= \mu_{w_0}\mu_{w_3}\mu_{w_0}\mu_{w_3}\mu_{w_0} \mu_{w_0}\mu_{w_2}\mu_{w_0}\mu_{w_2}\mu_{w_0}\mu_{w_0}\mu_{w_1}\mu_{w_0}\mu_{w_1}\mu_{w_0}=\mu_1\mu_4\mu_1\mu_4\mu_1\mu_1\mu_3\mu_1\mu_3\mu_1\mu_1\mu_2\mu_1\mu_2\mu_1.\]
 Let $(\mathbf{x}',Q')=\mu^{(4)}(\mathbf{x},Q)$, then $x_{i_4}$ is in the $\sigma^{-1}(i_4)$-th position in $\mathbf{x}'$.
 Indeed we have that $\mathbf{x}'=(x_4,x_1,x_2,x_3)$ and $Q'$ is given as follows.
 \[\begin{tikzpicture}[scale=1.3]
\node at (0,0) (11) {$2$};
\node at (1,0) (21) {$3$};
\node at (2,0) (31) {$4$};
\node at (3,0) (41) {$1$};
\path[-angle 90]
	(11) edge (21)
	(21) edge (31)
	(31) edge (41);
\end{tikzpicture}\]

 Step 2. Take $i_3=3$, then $v_3=4$ and $\sigma^{-1}(i_3)=2$. Let $s_0=2, s_1=3, s_2=4$,  thus \[\mu^{(3)}= (\mu_{s_0}\mu_{s_2}\mu_{s_0}\mu_{s_2}\mu_{s_0})(\mu_{s_0}\mu_{s_1}\mu_{s_0}\mu_{s_1}\mu_{s_0})=(\mu_2\mu_4\mu_2\mu_4\mu_2)(\mu_2\mu_3\mu_2\mu_3\mu_2).\]
 Let $(\mathbf{x}'',Q'')=\mu^{(3)}(\mathbf{x}',Q')$, then $x_{i_3}$ is in the $\sigma^{-1}(i_3)$-th position and $x_{i_4}$ is in the $\sigma^{-1}(i_4)$-th position in $\mathbf{x}''$.
 Indeed we have that $\mathbf{x}''=(x_4,x_3,x_1,x_2)$ and $Q''$ is given as follows.
 \[\begin{tikzpicture}[scale=1.3]
\node at (0,0) (11) {$3$};
\node at (1,0) (21) {$4$};
\node at (2,0) (31) {$2$};
\node at (3,0) (41) {$1$};
\path[-angle 90]
	(11) edge (21)
	(21) edge (31)
	(31) edge (41);
\end{tikzpicture}\]

 Step 3. Take $i_2=1$, then $v_2=3$ and $\sigma^{-1}(i_2)=4$. Let $p_0=4, p_1=3$, \[\mu^{(2)}= \mu_{p_0}\mu_{p_1}\mu_{p_0}\mu_{p_1}\mu_{p_0}=
 \mu_4\mu_3\mu_4\mu_3\mu_4.\]
 Let $(\mathbf{x}''',Q''')=\mu^{(2)}(\mathbf{x}'',Q'')$, then $x_{i_2}$ is in the $\sigma^{-1}(i_2)$-th position in $\mathbf{x}'''$.
 Indeed we have that $\mathbf{x}'''=(x_4,x_3,x_2,x_1)$ and $Q'''$ is given as follows.
 \[\begin{tikzpicture}[scale=1.3]
\node at (0,0) (11) {$4$};
\node at (1,0) (21) {$3$};
\node at (2,0) (31) {$2$};
\node at (3,0) (41) {$1$};
\path[-angle 90]
	(11) edge (21)
	(21) edge (31)
	(31) edge (41);
\end{tikzpicture}\]

Note that $(\mathbf{x}''',Q''')=(\mathbf{x},Q)^{\sigma}$, then we have that $(\mathbf{x},Q)^{\sigma}=\mu^{(2)}\mu^{(3)}\mu^{(4)}(\mathbf{x},Q)$.
\end{Example}

In particular, for exchange matrices of Dynkin type and Euclidean type, strict direct cluster automorphism
groups equal to direct cluster automorphism groups, and the latter have been computed in [\cite{ASS}, Table1], which
also supports a table on strict direct cluster automorphism groups.

\section{Application: Automorphism-finite cluster algebras}\label{4}
In this section, we study automorphism-finite cluster algebras. Recall that a cluster algebra $\mathcal{A}$ is called \textbf{automorphism-finite} if its cluster automorphism group  $\mathrm{Aut}(\mathcal{A})$ is a finite group, otherwise it is called \textbf{automorphism-infinite}.  In \cite{ASS},  Assem, Schiffler and Shamchenko proved a
cluster algebra with an acyclic skew-symmetric exchange matrix or from a surface is automorphism-finite
if and only if it is of Dynkin type. We mainly consider two cases of skew-symmetrizable exchange matrices,
one is cluster algebras with bipartite seeds, the other one is cluster algebras of finite mutation type.

For cluster algebras of type $A,D$ and $E$, their automorphism groups, listed in (\cite{ASS}, Table 3.3), are finite;
 for those of type $B,C,G$ and $F$, their automorphism groups are also finite, given in (\cite{CZ2}, Table 1) due to  (\cite{CZ2}, Theorem 3.5). In summary, we have:

\begin{Lemma}[\cite{ASS,CZ2}]\label{autofiDy}
A skew-symmetrizable cluster algebra of finite type is always  automorphism-finite.
\end{Lemma}

\begin{Lemma}\label{lem3} For a cluster algebra $\mathcal{A}=\mathcal{A}(\mathbf{x},B)$, if there is an $I$-sequence
$\mathbf{i}$ such that $\mathbf{i}$ is a period of $B$ and $\mathbf{i}^N$ is not a period of $(\mathbf{x},B)$
for any $N\in \mathbb{Z}_{>0}$, then $\mathcal{A}$ is automorphism-infinite.
\end{Lemma}
\begin{proof}
 Since $\mathbf{i}$ is a period of $B$£¬ we get $\mu_{\mathbf{i}}\in G(B)$  and then $\mu_{\mathbf{i}}H(\mathbf{x},B)\in G(B)/H(\mathbf{x},B)$. And since $\mathbf{i}^N$ is not a period of $(\mathbf{x},B)$ for any $N\in \mathbb{Z}_{>0}$, it follows that the order of $\mu_{\mathbf{i}}H(\mathbf{x},B)$ is infinite as a group element. Hence  the quotient group $G(B)/H(\mathbf{x},B)$ is an infinite group.
 By Theorem \ref{exa1}, $\mathrm{SAut}^+(\mathcal{A})\cong G(B)/H(\mathbf{x},B)$.
Hence,  $\mathrm{Aut}(\mathcal{A})$ is infinite by Corollary \ref{finitegp}, which means that $\mathcal{A}$ is automorphism-infinite.
\end{proof}

\begin{Definition}[\cite{FZ4}]
For a labeled seed $(\mathbf{x},B)$, if there is a function $\varepsilon : [1,n]\rightarrow \{-1,1\}$ such that
\[ \varepsilon(i)=1,\,\, \varepsilon(j)=-1,\, whenever\,\, b_{ij}>0,\]
the labeled seed and its exchange matrix is said to be \textbf{bipartite}.
\end{Definition}
Note that $b_{ij}=0$ implies $\mu_i\mu_j(\mathbf{x},B)=\mu_j\mu_i(\mathbf{x},B)$. This makes the following
compositions of mutation sequences well-defined.
\[ \mu_+=\prod_{\varepsilon(k)=1}\mu_k,\,\,\,\,\, \mu_-=\prod_{\varepsilon(k)=-1}\mu_k.\]

Note that $\mu_+$ and $\mu_-$ are involutions. Since $\mu_+(B)=-B, \mu_-(B)=-B$, $\mu_+$ and $\mu_-$ transform
bipartite seeds to bipartite seeds.

\begin{Definition}[\cite{FZ4}]\label{fintype}
For $s\in \mathbb{Z}_{>0}$, define a sequence of labeled seeds as follows.
\[(\mathbf{x}_s,(-1)^sB) = \underbrace{\mu_{\pm}\dots\mu_+\mu_-}_{s\, factors}(\mathbf{x},B),\]
\[(\mathbf{x}_{-s},(-1)^sB) = \underbrace{\mu_{\mp}\dots\mu_-\mu_+}_{s\, factors}(\mathbf{x},B).\]
We call the family $\{(\mathbf{x}_s,(-1)^sB)\}_{s\in \mathbb{Z}}$ a \textbf{bipartite belt}.
\end{Definition}

Assume that $\mathbf{x}_s=(x_{1;s},x_{2;s},\dots,x_{n;s})$. Fomin and Zelevinsky studied the following
family \begin{equation}\label{family}\{x_{i;s}:\varepsilon(i)=(-1)^s\}.\end{equation}

\begin{Theorem}[\cite{FZ4}]\label{thm3} Suppose that $B$ is an indecomposable bipartite skew-symmetrizable matrix.

(i) If $B$ is of finite type, then the corresponding bipartite belt satisfies
\[ (\mathbf{x}_s,(-1)^sB)=(\mathbf{x}_{s+2(h+2)},(-1)^{s+2(h+2)}B),\,\mbox{for\, any\,\,}m\in \mathbb{Z},\]
wehere $h$ is the corresponding Coxeter number.

(ii) If $B$ is not of finite type, then all the elements $x_{i;m}$ in (\ref{family}) are distinct
viewed as Laurent polynomials in the initial labeled seed.
\end{Theorem}

As an application of Lemma \ref{lem3}, we consider cluster algebras of finite mutation type, which have
been classified in \cite{FST1,FST2}.

\begin{Theorem}
 If a cluster algebra $\mathcal{A}=\mathcal{A}(\mathbf{x},B)$  is either of finite mutation
type or its initial exchange matrix $B$ is an indecomposable bipartite matrix, then
$\mathcal{A}$ is automorphism-finite if and only if $\mathcal{A}$ is of finite type.
\end{Theorem}
\begin{proof}

If $B$ is an indecomposable bipartite matrix, then it  follows from Lemma \ref{lem3} and Theorem \ref{thm3}.

Assume that $\mathcal{A}$ is of finite mutation type and is not of finite type. Then there is
a labeled seed $(\mathbf{x}',B')$ such that $|b'_{ij}b'_{ji}|>3$ for some $i,j\in [1,n]$.
Without loss of generality, we may assume that $|b_{12}b_{21}|>3$.
For $m\in Z_{>0}$, define a sequence of labeled seeds as follows.
\[(\mathbf{x}_m,B_m)=\underbrace{\mu_{1or2}\dots\mu_1\mu_2}_{m\,factors}(\mathbf{x},B),\]
\[(\mathbf{x}_{-m},B_{-m})=\underbrace{\mu_{2or1}\dots\mu_2\mu_1}_{m\,factors}(\mathbf{x},B).\]
Since $\mathcal{A}$ is of finite mutation type, there exists $k\in \mathbb{Z}_{>0}$ such that
$B_k=B$. Let $\mathbf{i}=\{1or2,\dots, 1, 2\}$ be the sequence with $k$ factors so that we have
$\mathbf{i}$ is a period of $B$. If there exists some $N\in \mathbb{Z}_{>0}$ such that
$\mathbf{i}^N$ is a period of $(\mathbf{x},B)$, by Restriction Theorem, $\mathbf{i}^N$ is a period of $(\hat{\mathbf{x}},\hat{B})$,
where $\hat{\mathbf{x}}=(x_1,x_2)$,and $\hat{B}=\begin{bmatrix}  0&b_{12}\\ b_{21} &0 \end{bmatrix}$ with $|b_{12}b_{21}|>3$.
This implies $\mathcal{A}(\hat{\mathbf{x}},\hat{B})$ is of finite type, which is a contradiction.
Then $\mathbf{i}^N$ is not a period of $(\mathbf{x},B)$ for any $N\in \mathbb{Z}_{>0}$.
By Lemma \ref{lem3}, $\mathcal{A}$ is not automorphism-finite.
\end{proof}
\begin{Remark}
There exist  cluster algebras of infinite mutation type whose cluster automorphism groups are finite.
In particular, they are not of  Dynkin type.  Indeed it is proved in \cite{W} that if $Q$ is a mutation-cyclic
$3$-point-quiver, then $G(Q)=1$ and it implies that the corresponding cluster algebra is automorphism-finite.
\end{Remark}

\section{Sufficient conditions for $\mathrm{Aut}\mathcal{A}\cong\mathrm{Aut}_{M_n}S$}\label{final}
In this section, we consider the relations between cluster automorphism group $\mathrm{Aut}\mathcal{A}$
and the automorphism group $\mathrm{Aut}_{M_n}S$ of a labeled mutation class $S$ which is defined in \cite{KP} via periodicities in cluster algebras.
Note that in the sequel we can only involve the cluster algebras in the skew-symmetric case.

\begin{Definition}[\cite{KP}] Given a labeled  seed $(\mathbf{x},B)$ of rank $n$ with $B$ is skew-symmetric.
\begin{enumerate}
\item [(1)] The global mutation group for labeled seeds of rank $n$ is given by
\[ M_n =S_n \ltimes \langle\mu_1, \mu_2, \dots, \mu_n : \mu_i^2=1\rangle  ,\]
where $\mu_i$ are mutations and $\mu_i\sigma=\sigma\mu_{\sigma(i)}$ for $\sigma\in S_n$.

\item [(2)] The labeled mutation class $S$ of $(\mathbf{x},B)$ is the orbit of $(\mathbf{x},B)$
under the action $M_n$.

\item [(3)] The automorphism group $\mathrm{Aut}_{M_n}(S)$ of $S$ consists of bijections from $S$ to itself which
commute with the action of $M_n$.

\item [(4)] We denote by $W$ the subgroup of $\mathrm{Aut}_{M_n}(S)$ defined as follows
 \[W=\{f\in \mathrm{Aut}_{M_n}(S)|\,if\,\, f(\mathbf{x},B)=(\mathbf{x}',B'),\,\, then\,\, B'=B\,\, or\,\, B'=-B\}.\]
\end{enumerate}
\end{Definition}

King and Pressland gave relations between these groups with cluster automorphism groups and propose a problem in \cite{KP}.

\begin{Theorem}[\cite{KP}]\label{KP1}
For any labeled  seed $(\mathbf{x},B)$, the automorphism group  $\mathrm{Aut}\mathcal{A}$ is isomorphic
to the subgroup $W$ of $\mathrm{Aut}_{M_n}(S)$.
In particular, if $B$ is of finite mutation type, then $W= \mathrm{Aut}_{M_n}(S)$ and thus
\[\mathrm{Aut}\mathcal{A}\cong \mathrm{Aut}_{M_n}(S).\]
\end{Theorem}

\begin{Problem}[\cite{KP}]\label{prob}
Does the property $W= \mathrm{Aut}_{M_n}(S)$ characterise the finite mutation
type?
\end{Problem}

 Now, we answer this question negatively by finding a number of cluster algebras which are not of finite mutation type,
however $W= \mathrm{Aut}_{M_n}(S)$ holds. First, we have the following lemma:

\begin{Lemma}\label{le0}
Let $(\mathbf{x},B)$ be an arbitrary labeled seed of rank $n$, then the following statements hold.
\begin{enumerate}
\item[(i)](\cite{Nak}) The labeled seed $(\mathbf{x}',-B)$ has the same periods with $(\mathbf{x},B)$, i.e., $P(\mathbf{x},B)=P(\mathbf{x}',-B)$,
\item[(ii)] If $P(\mathbf{x},B)=P(\mathbf{x}',B')$ for some labeled seed $(\mathbf{x}',B')$ of rank $n$,
then \[P(\mu_{\mathbf{i}}(\mathbf{x},B))=P(\mu_{\mathbf{i}}(\mathbf{x}',B'))\] for any $I$-sequence $\mathbf{i}$.
\end{enumerate}
\end{Lemma}
\begin{proof}
For (ii), it is enough to prove that
$$P(\mu_{\mathbf{i}}(\mathbf{x},B))\subset P(\mu_{\mathbf{i}}(\mathbf{x}',B')).$$

In fact, let $\mathbf{j}$, as an $I$-sequence, be a period of $\mu_{\mathbf{i}}(\mathbf{x},B)$, then we have
 $\mu_{\mathbf{j}}\mu_{\mathbf{i}}(\mathbf{x},B)=\mu_{\mathbf{i}}(\mathbf{x},B)$, thus
$\mu_{\mathbf{i}}^{-1}\mu_{\mathbf{j}}\mu_{\mathbf{i}}(\mathbf{x},B)=(\mathbf{x},B)$; then since $P(\mathbf{x},B)=P(\mathbf{x}',B')$, we have
$\mu_{\mathbf{i}}^{-1}\mu_{\mathbf{j}}\mu_{\mathbf{i}}(\mathbf{x}',B')=(\mathbf{x}',B')$; as follows, we obtain
$\mu_{\mathbf{j}}\mu_{\mathbf{i}}(\mathbf{x}',B')=\mu_{\mathbf{i}}(\mathbf{x}',B')$.
Therefore $\mathbf{j}$ is also a period of $\mu_{\mathbf{i}}(\mathbf{x},B)$.
\end{proof}

\begin{Lemma}\label{le1} Let $(\mathbf{x},B)$ and $(\mathbf{x}',B')$ be two labeled seeds of rank $n$.
Assume that $P(\mathbf{x},B)=P(\mathbf{x}',B')$, then for any $i,j\in [1,n]$,
the following statements hold:
 \begin{enumerate}
\item[(i)] $|b_{ij}|=s$ if and only if $|b_{ij}'|=s$ for $s=0, 1$;
\item[(ii)] $|b_{ij}|\geqslant 2$ if and only if $|b_{ij}'|\geqslant 2$.
\end{enumerate}
\end{Lemma}
\begin{proof}
It follows from Theorem \ref{RE} and Example \ref{rank2} that, for any labeled seed
$(\hat{\mathbf{x}}, \hat{B})$ of rank $2$, $H(\hat{\mathbf{x}}, \hat{B})=\{(\mu_1\mu_2)^{2m})|m\in \mathbb{Z}\}$
if and only if $|\hat{b}_{12}|=0$, $H(\hat{\mathbf{x}}, \hat{B})=\{(\mu_1\mu_2)^{5m})|m\in \mathbb{Z}\}$
if and only if $|\hat{b}_{12}|=1$, and $H(\hat{\mathbf{x}}, \hat{B})=\mathbf{1}$
if and only if $|\hat{b}_{12}|\geqslant 2$.
\end{proof}

\begin{Lemma}\label{le2} Let $(\mathbf{x},B)$ and $(\mathbf{x}',B')$ be two labeled seeds of rank $3$.
Suppose the corresponding quiver $Q$ of $B$ is acyclic and has the following form
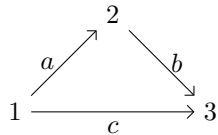
\[\mathord{\begin{tikzpicture}[scale=1.3,baseline=0]
\node at (0,0.5) (2) {$2$};
\node at (-1,-0.5) (1) {$1$};
\node at (1,-0.5) (3) {$3$};
\path[-angle 90]
	(1) edge node [left] {$a$} (2)
	(2) edge node [right] {$b$} (3)
	(1) edge node [below] {$c$} (3);
\end{tikzpicture}}\]
with $\mathrm{min}\{a,b,c\}=1$. If $P(\mathbf{x},B)=P(\mathbf{x}',B')$, then the quiver $Q'$
of $B'$ is also acyclic.
\end{Lemma}
\begin{proof}
Assume that $P(\mathbf{x},B)=P(\mathbf{x}',B')$, and $Q'$ is not acyclic.
Then the quiver $Q'$ is of the following form
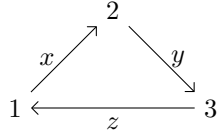
\[\mathord{\begin{tikzpicture}[scale=1.3,baseline=0]
\node at (0,0.5) (2) {$2$};
\node at (-1,-0.5) (1) {$1$};
\node at (1,-0.5) (3) {$3$};
\path[-angle 90]
	(1) edge node [left] {$x$} (2)
	(2) edge node [right] {$y$} (3)
	(3) edge node [below] {$z$} (1);
\end{tikzpicture}}\]
with $\mathrm{min}\{x,y,z\}=1$ by Lemma \ref{le1}.

If $c=1=z$, then $x\neq y$. Otherwise the number of arrows between $1$ and $2$ in $\mu_3(Q')$ is zero, however it is $a$ in $\mu_3(Q)$.
Then  $P(\mu_3(\mathbf{x},B))\neq P(\mu_3(\mathbf{x}',B'))$ which is a contradiction by Lemma \ref{le0}(ii). We may assume that
$y>x\geqslant 1$. Then the numbers of arrows between $1$ and $3$ in $\mu_2\mu_3(Q)$ and $\mu_2\mu_3(Q')$ are $1$ and $y(y-x)-1$, respectively.
Thus $y(y-x)-1=1$ and this implies $x=1,y=2$. However the numbers of arrows between $2$ and $3$ in $\mu_1(Q)$ and $\mu_1(Q')$ are $b(\geqslant 2)$
and $1$, which is a contradiction. The case $x>y(\geqslant 1)$ is similar. Thus $z\geqslant 2$.

If $a=1=x$, the numbers of arrows between $1$ and $2$ in $\mu_3(Q)$ and $\mu_3(Q')$ are $1$ and $yz-1$, respectively. Thus $yz-1=1$ and this
implies $y=1,z=2$ since $z\geqslant 2$. In this case, the numbers of arrows between $1$ and $3$ in $\mu_2(Q)$ and $\mu_2(Q')$ are
$y+1(\geqslant 2)$ and $1$, which is a contradiction. Thus $x\geqslant 2$.

If $b=1=y$, the same discussion on $Q^{op}$ and $Q'^{op}$ as in the case $a=1=x$ implies $y\geqslant 2$, which is a contradiction. Therefore
$Q'$ cannot be cyclic.
 \end{proof}

The following two lemmas and Corollary \ref{pr} are the improvement of Lemma 6.9 in \cite{KP}.
\begin{Lemma}\label{le3}
Let $(\mathbf{x},B)$ and $(\mathbf{x}',B')$ be two labeled seeds of rank $3$. Suppose that
the quiver $Q$ of $B$ is one of the following  forms
\[\begin{tikzpicture}[scale=1.3]
\node at (0,0) (11) {$1$};
\node at (1,0) (21) {$2$};
\node at (2,0) (31) {$3$};
\node at (2.5,0) (or) { and };
\node at (3,0) (12) {$1$};
\node at (4,0) (22) {$2$};
\node at (5,0) (32) {$3$};
\path[-angle 90]
	(11) edge node [above] {$a$} (21)
	(21) edge node [above] {$b$} (31)
	(22) edge node [above] {$a$} (12)
	(32) edge node [above] {$b$} (22);
\end{tikzpicture}\]
and the quiver $Q'$ of $B'$ is one of the following forms:
\[\begin{tikzpicture}[scale=1.3]
\node at (0,0) (11) {$1$};
\node at (1,0) (21) {$2$};
\node at (2,0) (31) {$3$};
\node at (2.5,0) (or) { and };
\node at (3,0) (12) {$1$};
\node at (4,0) (22) {$2$};
\node at (5,0) (32) {$3$};
\path[-angle 90]
	(11) edge node [above] {$x$} (21)
	(31) edge node [above] {$y$} (21)
	(22) edge node [above] {$x$} (12)
	(22) edge node [above] {$y$} (32);
\end{tikzpicture}\]
where $a, b, x, y\in \mathbb{Z}_{>0}$. Then $P(\mathbf{x},B)\neq P(\mathbf{x}',B')$.
\end{Lemma}

\begin{proof}
The weights of arrows between $1$ and $3$ in $\mu_2(Q)$ and $\mu_2(Q')$ are $ab$ and $0$, respectively. Therefore
$P(\mu_2(\mathbf{x},B))\neq P(\mu_2(\mathbf{x}',B'))$ by Lemma \ref{le1}, which also implies
$P(\mathbf{x},B)\neq P(\mathbf{x}',B')$ by Lemma \ref{le0}(ii).
\end{proof}

\begin{Lemma}\label{le4}
Let $(\mathbf{x},B)$ and $(\mathbf{x}',B')$ be two labeled seeds of rank $3$. Suppose that
the quiver $Q$ of $B$ is one of the following  forms:
\[\mathord{\begin{tikzpicture}[scale=1.3,baseline=0]
\node at (0,0.3) (12) {$2$};
\node at (-1,-0.3) (11) {$1$};
\node at (1,-0.3) (13) {$3$,};
\node at (2.5,0.3) (22) {$2$};
\node at (1.5,-0.3) (21) {$1$};
\node at (3.5,-0.3) (23) {$3$,};
\node at (5,0.3) (32) {$2$};
\node at (4,-0.3) (31) {$1$};
\node at (6,-0.3) (33) {$3$,};
\node at (6.5,0) (or) {and};
\node at (8,0.3) (42) {$2$};
\node at (7,-0.3) (41) {$1$};
\node at (9,-0.3) (43) {$3$,};
\path[-angle 90]
	(11) edge node [left] {$a$} (12)
	(12) edge node [right] {$b$} (13)
	(11) edge node [below] {$c$} (13)
	(21) edge node [left] {$a$} (22)
	(22) edge node [right] {$b$} (23)
	(23) edge node [below] {$c$} (21)
	(32) edge node [left] {$a$} (31)
	(33) edge node [right] {$b$} (32)
	(33) edge node [below] {$c$} (31)
	(42) edge node [left] {$a$} (41)
	(43) edge node [right] {$b$} (42)
	(41) edge node [below] {$c$} (43);
\end{tikzpicture}}\]
and quiver $Q'$ of $B'$ is one of the following forms:
\[\mathord{\begin{tikzpicture}[scale=1.3,baseline=0]
\node at (0,0.3) (12) {$2$};
\node at (-1,-0.3) (11) {$1$};
\node at (1,-0.3) (13) {$3$,};
\node at (2.5,0.3) (22) {$2$};
\node at (1.5,-0.3) (21) {$1$};
\node at (3.5,-0.3) (23) {$3$,};
\node at (5,0.3) (32) {$2$};
\node at (4,-0.3) (31) {$1$};
\node at (6,-0.3) (33) {$3$,};
\node at (6.5,0) (or) {and};
\node at (8,0.3) (42) {$2$};
\node at (7,-0.3) (41) {$1$};
\node at (9,-0.3) (43) {$3$,};
\path[-angle 90]
	(11) edge node [left] {$x$} (12)
	(13) edge node [right] {$y$} (12)
	(11) edge node [below] {$z$} (13)
	(21) edge node [left] {$x$} (22)
	(23) edge node [right] {$y$} (22)
	(23) edge node [below] {$z$} (21)
	(32) edge node [left] {$x$} (31)
	(32) edge node [right] {$y$} (33)
	(33) edge node [below] {$z$} (31)
	(42) edge node [left] {$x$} (41)
	(42) edge node [right] {$y$} (43)
	(41) edge node [below] {$z$} (43);
\end{tikzpicture}}\]
where $a,b,c,x,y,z\in \mathbb{Z}_{>0}$, and $\mathrm{min}\{a,b,c\}=1$. Then $P(\mathbf{x},B)\neq P(\mathbf{x}',B')$.
\end{Lemma}
\begin{proof}
Assume that $\mathrm{min}\{x,y,z\} = 1$, $x=1 (y=1,z=1,\mathrm{resp}.)$ if and only if $a=1 (b=1,c=1,\mathrm{resp}.)$,
and $x>1 (y>1,z>1,\mathrm{resp}.)$ if and only if $a>1 (b>1,c>1,\mathrm{resp}.)$. Otherwise it is obvious that
$P(\mathbf{x},B)\neq P(\mathbf{x}',B')$ by Lemma \ref{le1}. By Lemma \ref{le0}(i) and Lemma \ref{le2},
it is enough to consider that $Q$ is of the following form
\[\mathord{\begin{tikzpicture}[scale=1.3,baseline=0]
\node at (0,0.3) (12) {$2$};
\node at (-1,-0.3) (11) {$1$};
\node at (1,-0.3) (13) {$3$,};
\path[-angle 90]
	(11) edge node [left] {$a$} (12)
	(12) edge node [right] {$b$} (13)
	(11) edge node [below] {$c$} (13);
\end{tikzpicture}}\]
and $Q'$ is one of the following forms
\[\mathord{\begin{tikzpicture}[scale=1.3,baseline=0]
\node at (0,0.3) (12) {$2$};
\node at (-1,-0.3) (11) {$1$};
\node at (1,-0.3) (13) {$3$,};
\node at (1.5,0) (and) {and};
\node at (3,0.3) (22) {$2$};
\node at (2,-0.3) (21) {$1$};
\node at (4,-0.3) (23) {$3$,};
\path[-angle 90]
	(11) edge node [left] {$x$} (12)
	(13) edge node [right] {$y$} (12)
	(11) edge node [below] {$z$} (13)
	(21) edge node [left] {$x$} (22)
	(23) edge node [right] {$y$} (22)
	(23) edge node [below] {$z$} (21);
\end{tikzpicture}}\]

If $c=z=1$, the weights of arrows between $1$ and $3$ in $\mu_2(Q)$ and $\mu_2(Q')$ are $ab+1$ and $1$, respectively.
Then $P(\mu_2(\mathbf{x},B))\neq P(\mu_2(\mathbf{x}',B'))$ and thus $P(\mathbf{x},B)\neq P(\mathbf{x}',B')$.

Assume that
$c>1, z>1$ in the following. We denote by $Q_1$, $Q_2$, and $Q_3$ the above three quivers in order for simplicity. Note that $Q_1=Q$ and
$Q_2$ and $Q_3$ are the two possible forms of $Q'$.

If $a=x=1$, the weights of arrows between $1$ and $2$ in $\mu_3(Q_1)$ and $\mu_3(Q_2)$ are $1$ and $yz+1$, respectively.
This implies $P(\mathbf{x},B)\neq P(\mathbf{x}',B')$ when $Q'=Q_2$. Suppose that $Q'=Q_3$. Then the weights of arrows between
$1$ and $2$ are $c(b+c)-1$ and $1$ in $\mu_3\mu_1\mu_2(Q)$ and $\mu_3\mu_1\mu_2(Q')$, respectively. Since $c>1,b\geqslant 1$,
thus $c(b+c)-1>1$, which implies $P(\mathbf{x},B)\neq P(\mathbf{x}',B')$.

Finally assume that $a>1,x>1$, and $b=y=1$. The weights of arrows between $2$ and $3$ in $\mu_1(Q_1)$ and $\mu_1(Q_3)$
are $1$ and $xz+1$, respectively. This $P(\mathbf{x},B)\neq P(\mathbf{x}',B')$ when $Q'=Q_3$. Suppose that $Q'=Q_2$.
Then the weights of arrows between
$2$ and $3$ are $c(a+c)-1$ and $1$ in $\mu_1\mu_3\mu_2(Q)$ and $\mu_1\mu_3\mu_2(Q')$, respectively. Since $c>1,a>1$,
thus $c(a+c)-1>1$, which implies $P(\mathbf{x},B)\neq P(\mathbf{x}',B')$.
\end{proof}

\begin{Remark}
In Lemmas \ref{le0}-\ref{le4}, two labeled seeds $(\mathbf{x},B)$ and $(\mathbf{x}',B')$ are not assumed
to be mutation equivalent in our settings.
\end{Remark}

For the corresponding quiver $Q=Q(B)$ of $B$ with only three vertices $i, j, k$, we call a vertex $i$ an \textbf{inflexion} if $b_{ji}b_{ik}>0$.
As a conclusion of Lemma \ref{le3} and \ref{le4}, we have the following corollary.

\begin{Corollary}\label{pr}
Let $(\mathbf{\tilde{x}},\tilde{B})$ and $(\mathbf{\tilde{x}}',\tilde{B}')$ be two labeled seeds of rank $n$
satisfying that $P(\mathbf{\tilde{x}},\tilde{B})= P(\mathbf{\tilde{x}}',\tilde{B}')$.
Let $(\mathbf{x},B)$ and $(\mathbf{x}',B')$ be the corresponding full subseed of
 $(\mathbf{\tilde{x}},\tilde{B})$ and $(\mathbf{\tilde{x}}',\tilde{B}')$ indexed by $I=\{i,j,k\}$ with $1\leq i<j<k\leq n$ respectively.
Let $r= min\{|b_{ij}|,|b_{jk}|,|b_{ik}|\}$. If $r\in \{0,1\}$, then a vertex $v$ is an inflexion  in $Q$ if and only if $v$ is an inflexion
in $Q'$ for any $v\in Q_0=Q'_0=\{i,j,k\}$, where $Q$ and $Q'$ are the corresponding quiver of $B$ and $B'$ respectively.

In particular, if $min\{|\tilde{b}_{ij}|,|\tilde{b}_{jk}|,|\tilde{b}_{ik}|\}\in \{0, 1\}$  for any $1\leq i<j<k\leq n$
 such that the full subquiver determined by $\{i,j,k\}$ of $\tilde{Q}$ is connected,
 then either the signs of $\tilde{b}_{xy}$ and $\tilde{b}_{xy}'$ are the same
for all $x,y\in [1,n]$, or the signs of $\tilde{b}_{ij}$ and $\tilde{b}_{ij}'$ are opposite for all $i,j\in [1,n]$.
\end{Corollary}
\begin{proof}
It follows from the restriction part of Theorem \ref{RE} that $P(\mathbf{x},B)=P(\mathbf{x}',B')$.

We may, without loss of generality, assume that $i=1,j=2,k=3$ and $Q$ is connected.
If $r=0$, by Lemma 5.7, the vertex $2$ is an inflexion  in $Q$  if and only if $2$ is an inflexion  in $Q'$. For $v=1,3$, it also holds
if we change the indices in Lemma 5.7.

If $r=1$, by Lemma 5.8, the vertex $2$ is an inflexion  in $Q$  if and only if $2$ is an inflexion  in $Q'$. For $v=1,3$, it also holds
if we change the indices in Lemma 5.8.

In particular, the whole orientation of a quiver with three vertices is determined by the orientation of one arrow and inflexions of all vertices,
thus it follows easily that either $Q$ and $Q'$ have the same orientations or $Q$ and $Q'$ have opposite orientations.
By the indecomposability of $B$ and $B'$, we have that either the signs of $\tilde{b}_{xy}$ and $\tilde{b}_{xy}'$ are the same
for all $x,y\in [1,n]$, or the signs of $\tilde{b}_{xy}$ and $\tilde{b}_{xy}'$ are the opposite for all $x,y\in [1,n]$.
\end{proof}

\begin{Proposition}[\cite{FWZ,CK}]\label{acy}
Let $Q$ and $Q'$ be two acyclic quivers which are mutation equivalent to each other.
Then $Q$ can be transformed into a quiver isomorphic to $Q'$ via a sequence of mutations
at sources and sinks. Therefore all acyclic quivers in a given mutation class have the
same underlying undirected graph.

\end{Proposition}

Let us recall that if $B$ is a skew-symmetric matrix, then
\[m(B)= \mathrm{inf}\{v(B')|B'\, \mbox{is mutation equivalent to}\, B \},\]
where $v(B')=\mathrm{max}\{b'_{ij}|i,j\in [1,n]\}$.

\begin{Theorem}\label{main1}
Let $(\mathbf{x},B)$ be a labeled seed of rank $n$, $S$ be its labeled mutation class, and $\mathcal{A}$ be its corresponding cluster algebra.
If the exchange matrix $B$ satisfies one of the following conditions
\begin{enumerate}
\item[(i)] $m(B)=1$;
\item[(ii)] $Q(B)$ is acyclic, $v(B)=2$ and the underlying graph of $Q(B)$ has no $3$-cycles;
\item[(iii)] $Q(B)$ is acyclic, $v(B)=2$ and every $3$-cycle in the underlying graph of $Q(B)$ has at least one simple edge,
\end{enumerate}
then $W= \mathrm{Aut}_{M_n}(S)$. Consequently, in these cases, $\mathrm{Aut}\mathcal{A}\cong \mathrm{Aut}_{M_n}(S)$.
\end{Theorem}
\begin{proof}
Suppose that $f\in \mathrm{Aut}_{M_n}(S)$
and let $(\mathbf{x}',B')=f(\mathbf{x},B)$.  In order to prove $W\supseteq Aut_{M_n}(S)$, it is enough to prove that $B'=B$ or $B'=-B$.

Let $\mathbf{i}$ be a period of $(\mathbf{x}, B)$. Since $f$ is a bijection from $S$ to itself that commutes with the action of
$M_n$,  we have $(\mathbf{x}', B')=f(\mathbf{x}, B)=f(\mu_{\mathbf{i}}(\mathbf{x}, B))=\mu_{\mathbf{i}}(f(\mathbf{x},B))=\mu_{\mathbf{i}}(\mathbf{x}', B')$, which implies that
$\mathbf{i}$ is also a period of $(\mathbf{x}', B')$. Thus,
$P(\mathbf{x},B)\subset P(\mathbf{x}',B')$.

On the other hand, $(\mathbf{x}, B)=f^{-1}(\mathbf{x}', B')$ and $f^{-1}\in \mathrm{Aut}_{M_n}(S)$, similarly, we have $P(\mathbf{x}',B')\subset P(\mathbf{x},B)$. Hence,  $P(\mathbf{x},B)=P(\mathbf{x}',B')$.

Moreover, it follows from the restriction part of Theorem \ref{RE} that $P(\acute{\mathbf{x}},\acute{B})=P(\acute{\mathbf{x}}',\acute{B}')$
for any $3\times 3$ full subseeds $(\acute{\mathbf{x}},\acute{B})$ and $(\acute{\mathbf{x}}',\acute{B}')$ of $(\mathbf{x}, B)$ and $(\mathbf{x}', B')$ respectively.

In the case (i), without loss of generality, we may assume that $v(B)=1$, that is,
$|b_{ij}|\leqslant 1$ for any $i,j\in [1,n]$.
It follows from Lemma \ref{le1} that  $|b_{ij}| = |b_{ij}'|$ for any $i,j\in [1,n]$.  And it follows from Corollary \ref{pr} that either $B'=B$ or $B'=-B$, since each arrow between $i$ and $j$ may be considered in a connected full sub-quiver with three vertices.

In the case (ii),
 since the underlying graph of $Q(B)$ has no $3$-cycles, the underlying graph of every
 connected full subquiver with three vertices of $Q$  is of one of the forms listed in Lemma \ref{le3} up to permutations.
 By Corollary \ref{pr}, we have that either $b_{ij}$ and $b'_{ij}$ have the same
signs for all $i,j\in [1,n]$, or $b_{ij}$ and $b'_{ij}$ have opposite
signs for all $i,j\in [1,n]$. Thus the quiver $Q'$ of
$B'$ is also acyclic.

Since $(\mathbf{x}',B')=f(\mathbf{x},B)$, it means that $Q$ and $Q'$ are in the same labeled mutation class. Then,
 by Proposition \ref {acy},  $Q$ and $Q'$ have the same underlying graph up to permutations,
 which follows that $v(B')=v(B)=2$. By Lemma \ref{le1}, we have $|b_{ij}| = |b_{ij}'|$ for any $i,j\in [1,n]$. Therefore $B'=B$ or $B'=-B$.

In the case (iii), we claim that the quiver $Q'$ of $B'$ is also acyclic. If $Q'$ is not acyclic and it contains
 a directed $3$-cycle $\check{Q}'$. Consider the full subquiver determined
by the $3$-cycle and the corresponding full subquiver $\check{Q}$ in $Q$.
It follows from Theorem \ref{RE} that $P(\check{\mathbf{x}},\check{Q})=P(\check{\mathbf{x}}',\check{Q}')$. Since $\check{Q}$ is acyclic
and the minimal weight of $\check{Q}$ is $1$,
then by Lemma \ref{le2}, $\check{Q}'$ is also acyclic, which is a contradiction.
If $Q'$ contains no directed $3$-cycle, then there is a directed chordless $m$-cycle $\hat{Q}'$ with $m\geqslant 4$,  where a chordless $m$-cycle is a graph with $m$ vertices $v_{[1]},\dots, v_{[m]}$ which $[k]=\mathbb{Z}/k\mathbb{Z}$ for each $k\in[1,m]$ such that the number of edges between $v_{[i]}$ and $v_{[i+1]}$ is larger than zero and the number of edges between other vertices are zeros.
Consider the full subquiver determined by the chordless $m$-cycle and the corresponding full subquiver $\hat{Q}$ in $Q$.
Since $P(\mathbf{x},B)=P(\mathbf{x}',B')$, it follows from Theorem \ref{RE} that
$P(\hat{\mathbf{x}},\hat{Q})=P(\hat{\mathbf{x}}',\hat{Q}')$. And it follows from Lemma \ref{le3} that $\hat{Q}$ and $\hat{Q}'$
have the same or opposite orientations, which is a contradiction. Therefore $Q'$ has no directed cycles, i.e., $Q'$ is also
acyclic. Similarly since $v(B)=2$, we have $|b_{ij}| = |b_{ij}'|$ for any $i,j\in [1,n]$ by Lemma \ref{le1}
and Proposition \ref{acy}. And the orientation of $Q'$ is the same as $Q$ or opposite to $Q'$ by Corollary \ref{pr}, which implies $B'=B$ or $B'=-B$.
\end{proof}

Due to this theorem, we give the following examples as a negative answer to the King and Pressland's problem, i.e., Problem \ref{prob}.
Note that a skew-symmetric matrix $B$ of order at least 3 is mutation-finite if and only if $v(B')=\mathrm{max}\{b'_{ij}|i,j\in [1,n]\}\leq 2$
for any matrix $B'$ mutation equivalent to $B$, see [\cite{FST1}, Theorem 2.6].
\begin{enumerate}
\item [(1)] Let  $B=\begin{pmatrix} 0&1&1&1\\-1&0&1&0\\-1&-1&0&-1\\-1&0&1&0 \end{pmatrix}$. This $B$  satisfies the condition of Theorem \ref{main1}(i). However, we have
 $v(\mu_2\mu_4(B))=3$.

\item [(2)] Let $B=\begin{pmatrix} 0&2&0\\-2&0&1\\0&-1&0 \end{pmatrix}$. This $B$ satisfies the condition of Theorem \ref{main1}(ii). However, we have $v(\mu_1\mu_2(B))=3$.
\end{enumerate}

 So by the above note, these two $B$'s in (1) and (2)  are of mutation-infinite type. In the meantime, by Theorem \ref{main1}, we have $\mathrm{Aut}\mathcal{A}\cong W= \mathrm{Aut}_{M_n}(S)$. Hence, (1) and (2) give two counter-examples of  Problem \ref{prob}.

\begin{Remark}
(1) For a labeled seed $(\mathbf{x},B)$ with $B$ skew-symmetric, let $S$ be its labeled mutation class,
$\mathcal{A}$ be its corresponding cluster algebra, and $E$ be its exchange graph (see [\cite {FZ4}, Definition 4.2]).
 For the automorphism  group $\mathrm{Aut}(E)$ of $E$ as graph, Lawson proved that $\mathrm{Aut}_{M_n}(S)\cong \mathrm{Aut}(E)$ in \cite{Law}.
Thus the result of Theorem \ref{main1} also provides some sufficient conditions for $\mathrm{Aut}(\mathcal{A})\cong \mathrm{Aut}(E)$.

(2) For skew-symmetrizable case, in \cite{Law} Lawson had claimed that $\mathrm{Aut}(E)$, thus $\mathrm{Aut}_{M_n}(S)$,
is larger than $\mathrm{Aut}(\mathcal{A})$  for general skew-symmetrizable cluster algebras of finite mutation type, which is also the reason
we only consider skew-symmetric cluster algebras in this section.
\end{Remark}

\vspace{10mm}

{\bf Acknowledgements }\;
{\em This project was supported by the National Natural Science Foundation of China(No.11671350) and the Zhejiang Provincial Natural Science Foundation of China (No.LY19A010023).}

\end{document}